\documentclass[english,leqno]{article}
\usepackage[T1]{fontenc}
\usepackage[latin9]{inputenc}
\usepackage{array}
\usepackage{multirow}
\usepackage{amsthm}
\usepackage{amsmath}
\usepackage{amssymb}
\usepackage[numbers]{natbib}

\makeatletter

\providecommand{\tabularnewline}{\\}

\numberwithin{equation}{section}
\newcommand{\lyxaddress}[1]{
\par {\raggedright #1
\vspace{1.4em}
\noindent\par}
}
\theoremstyle{plain}
\newtheorem{thm}{\protect\theoremname}[section]
  \theoremstyle{definition}
  \newtheorem{defn}[thm]{\protect\definitionname}
  \theoremstyle{plain}
  \newtheorem{conjecture}[thm]{\protect\conjecturename}
  \theoremstyle{definition}
  \newtheorem{example}[thm]{\protect\examplename}
  \theoremstyle{remark}
  \newtheorem{note}[thm]{\protect\notename}
  \theoremstyle{remark}
  \newtheorem{rem}[thm]{\protect\remarkname}
  \theoremstyle{plain}
  \newtheorem{lem}[thm]{\protect\lemmaname}
  \theoremstyle{plain}
  \newtheorem{prop}[thm]{\protect\propositionname}
  \theoremstyle{plain}
  \newtheorem{cor}[thm]{\protect\corollaryname}
  \theoremstyle{remark}
  \newtheorem*{acknowledgement*}{\protect\acknowledgementname}

\usepackage{lscape}

\makeatother

\usepackage{babel}
  \providecommand{\acknowledgementname}{Acknowledgement}
  \providecommand{\conjecturename}{Conjecture}
  \providecommand{\corollaryname}{Corollary}
  \providecommand{\definitionname}{Definition}
  \providecommand{\examplename}{Example}
  \providecommand{\lemmaname}{Lemma}
  \providecommand{\notename}{Note}
  \providecommand{\propositionname}{Proposition}
  \providecommand{\remarkname}{Remark}
\providecommand{\theoremname}{Theorem}

\begin{document}

\title{A note on the Eisenbud-Mazur Conjecture}

\author{Ajinkya A More}

\maketitle

\lyxaddress{Alumnus, Department of Mathematics, University of Michigan, Ann Arbor}
\begin{abstract}
The Eisenbud-Mazur conjecture states that given an equicharacteristic
zero, regular local ring $(R,\mathfrak{m})$ and a prime ideal $P\subset R$,
we have that $P^{(2)}\subseteq mP$. In this paper, we computationally
prove that the conjecture holds in the special case of certain prime
ideals in formal power series rings. 
\end{abstract}
Eisenbud-Mazur conjecture, regular local rings, symbolic powers, evolutions

\section{Introduction}

Eisenbud and Mazur \citep{eisenbud-mazur} studied symbolic powers
in connection with the question of existence of non-trivial evolutions. 
\begin{defn}
Let $R$ be a ring and $S$ be a local $R$-algebra essentially of
finite type. An \textit{evolution}\textbf{ }of $S$ over $R$ consists
of the following data:\end{defn}
\begin{itemize}
\item A local $R$-algebra $T$ essentially of finite type.
\item A surjection $T\to S$ of $R$-algebras such that if $\Omega_{T/R}$
and $\Omega_{S/R}$ denote the modules of K\"ahler differentials
of $T$ over $R$ and of $S$ over $R$ respectively, then, the induced
map $\Omega_{T/R}\otimes_{T}S\to\Omega_{S/R}$ is an isomorphism. 
\end{itemize}
The evolution is said to be \emph{trivial}\textbf{ }if $T\to S$ is
an isomorphism.

The question of existence of non-trivial evolutions leads to the Eisenbud-Mazur
conjecture via Theorem \ref{thm:-connection to evolutions}. We first
need a more general definition of symbolic powers 
\begin{defn}
Let $R$ be a ring and $I$ an ideal in $R$. For a positive integer
$n$, the \emph{$n$th symbolic power }of $I$ is defined to be 
\[
I^{(n)}:=\{r\in R:\frac{r}{1}\in I^{n}R_{P}\text{ for all }P\text{ such that }P\text{ is an associated prime of }I\}.
\]

\end{defn}

\begin{defn}
Let $R$ be a ring and $I$ an ideal in $R$. We say $I$ is an \emph{unmixed
ideal}%
\footnote{\emph{Note that for the case of unmixed ideals, Definition 1.2 of
symbolic powers reduces to the definition that appears in }\citep{eisenbud-mazur}:
given a ring $R$, an ideal $I$ and a positive integer $n$, $I^{(n)}:=\{r\in R:\frac{r}{1}\in I^{n}R_{P}\text{ for all }P\text{ such that }P\text{ is an minimal prime of }I\}$%
} if every associated prime ideal of $I$ is isolated. In other words,
an unmixed ideal has no embedded prime ideals. \end{defn}
\begin{thm}
\label{thm:-connection to evolutions} \citep{eisenbud-mazur} Let
$R$ be a regular ring. Let $(P,\mathfrak{m})$ be a localization
of a polynomial ring in finitely many variables over $R$. Let $I$
be an ideal of $P$. If $P/I$ is reduced and generically separable
over $R$, then, every evolution of $P/I$ is trivial if and only
if $I^{(2)}\subseteq\mathfrak{m}I$.
\end{thm}

We now state a slightly more general version of the Eisenbud-Mazur
conjecture.
\begin{conjecture}
(Eisenbud-Mazur) \label{(Eisenbud-Mazur) conjecture} Given a regular
local ring $(R,\mathfrak{m})$ containing a field of characteristic
zero and an unmixed ideal $I$ in $R$, $I^{(2)}\subseteq\mathfrak{m}I$. 
\end{conjecture}

The hypothesis that $R$ be regular is necessary. If $R$ is not regular,
there exists a prime ideal $P$ in $R$ for which $P^{(2)}\nsubseteq\mathfrak{m}P$
as the following well-known example that appears in \citep{huneke-ribbe}
shows.
\begin{example}
Let $R=k[x,y,z]/(x^{2}-yz)$, $P=(x,y)R$. Then $z\in R\setminus P$
and $zy=x^{2}\in P^{2}$. So that, $y\in P^{(2)}$. However, $y\notin(x,y,z)P$.
\end{example}

Hochster and Huneke \citep{hochster-huneke-uniform-symbolic} showed
that if $(R,\mathfrak{m})$ is a regular local ring containing a field
and $P$ is a prime ideal of height $c$ then $P^{(c+1)}\subseteq\mathfrak{m}P$.

Eisenbud and Mazur construct examples in every positive characteristic
$p$ to show that the corresponding statement of Conjecture \ref{(Eisenbud-Mazur) conjecture}
does not hold.

\section{Problem definition}

The primary result of the paper is to prove the Eisenbud-Mazur conjecture
for the following case.
\begin{thm}
\label{main theorem}(Main theorem) Let $R=k[[t,x_{1},...,x_{m}]]$
be the formal power series ring over a field $k$ of characteristic
$0$, in $m+1$ indeterminates. Let $S=k[[t^{2},x_{1},...,x_{m}]]$.
Let $f_{1}(t),...,f_{m}(t)\in k[[t]]$. Let $Q_{1}=(x_{1}-f_{1}(t),...,x_{m}-f_{m}(t))R$
and $Q_{2}=(x_{1}-f_{1}(-t),...,x_{m}-f_{m}(-t))R$. Then $Q_{1},Q_{2}$
are prime ideals that are conjugate under the action of the automorphism
on $ $$R$ given by $\sigma:R\to R$, where $\sigma(t)=-t$ and $\sigma(x_{i})=x_{i}$
for $1\leq i\leq m$. Further, $Q_{1},Q_{2}$ contract to the same
prime ideal, say $P$ in $S$, \emph{i.e.,} $P=Q_{1}\cap S=Q_{2}\cap S=(Q_{1}\cap Q_{2})\cap S$.
Let $\mathfrak{m}=(t^{2},x_{1},...,x_{n})S$. Then $P^{(2)}\subseteq\mathfrak{m}P$.\end{thm}
\begin{note}
\label{special hypothesis} The conditions above are not as special
as it may seem. We show, in Proposition \ref{prop:reduction to dimension one}
below, that for an equicharacteristic complete local ring $S$, in
order to prove the Eisenbud-Mazur conjecture for prime ideals, it
is sufficient to restrict to prime ideals $P$ such that $\text{dim}(S/P)=1$.
Further, in Proposition \ref{prop: hypothesis is not special}, we
show that if $S=k[[t,x_{1},...,x_{m}]]$, where $k$ is an algebraically
closed field of characteristic $0$ and $P$ is a prime ideal in $S$
such that $\text{dim}(S/P)=1$, then, there exists a positive integer
$n$ such that for $R=k[[t^{\frac{1}{n}},x_{1},...,x_{m}]]$, there
exists a prime ideal $Q=(x_{1}-f_{1}(t^{\frac{1}{n}}),...,x_{m}-f_{m}(t^{\frac{1}{n}}))R$
such that $P=Q\cap S$. The case discussed in the preceding paragraph
is special of this set-up with $n=2$.\end{note}
\begin{rem}
From the start of section \ref{section with notation change} to the
remainder of the paper, the notation (including the assumptions on
the underlying rings and fields) will be as per the statement of Theorem
\ref{main theorem}. Prior to the start of section \ref{section with notation change},
any such assumptions that are made will be explicitly specified in
the statements of the correponding propositions.
\end{rem}
We first show that if $(R,\mathfrak{m})$ is an equicharacteristic
complete local ring and if the Eisenbud-Mazur conjecture holds for
height unmixed ideals $I$ such that $\text{dim}(R/I)=1$, then, it
holds for all height unmixed ideals in $R$. We first need a few preparatory
results starting with an irreducibility criterion for formal power
series (page 164, \citep{irreducible-power-series}). 
\begin{thm}
\citep{irreducible-power-series} \label{thm:-irreducibility criterion}
Consider the grading on $k[x,y]$ (where $k$ is a field) in which
$\text{deg}(x)=p>0$ and $\text{deg}(y)=q>0$. Let $R=k[[x,y]]$ and
let $f\in R\setminus\{0\}$. Let $l(f)$ denote the homogeneous polynomial
of smallest degree (with respect to the above grading) occurring in
$f$. If for some choice of $p,q$, $l(f)$ is an irreducible polynomial
in $k[x,y]$, then, $f$ is irreducible in $R$.\end{thm}
\begin{lem}
\label{lem:prime ideals in power series ring} Let $R=k[[x_{1},...,x_{n}]]$,
where $n>2$ and $k$ is a field. If $r,s$ are positive integers
such that $\text{gcd}(r,s)=1$, then, the ideal $P=(x_{1}^{r}-x_{2}^{s})R$
is prime in $R$.\end{lem}
\begin{proof}
Now $x_{1}^{r}-x_{2}^{s}$ generates a prime ideal in $k[x_{1},x_{2}]$
or equivalently it is irreducible when $ $$\text{gcd}(r,s)=1$. Then,
by Theorem \ref{thm:-irreducibility criterion}, $x_{1}^{r}-x_{2}^{s}$
is irreducible in $k[[x_{1},x_{2}]]$ and hence irreducible in $R$.
Then, since $R$ is a unique factorization domain, ideal $P=(x_{1}^{r}-x_{2}^{s})R$
is prime in $R$.\end{proof}
\begin{prop}
\label{prop:complete local rings have small primes} Let $(R,\mathfrak{m})$
be an equicharacteristic Noetherian complete local domain of dimension
$d\geq2$. Then, for every positive integer $n$, there exists a prime
ideal $P_{n}\neq0$ in $R$ such that $P_{n}\subseteq\mathfrak{m}^{n}$. \end{prop}
\begin{proof}
Since $(R,\mathfrak{m})$ is an equicharacteristic complete local
domain, it is module finite over $S=k[[x_{1},...,x_{d}]]$, where
$k$ is a field (Theorem 4.3.3, page 61, \citep{Hunke-Swanson}).
Since $d\geq2$, by lemma \ref{lem:prime ideals in power series ring},
$\mathfrak{p}_{n}=(x_{1}^{n}-x_{2}^{n+1})S$ is a prime ideal in $S$
and $\mathfrak{p}_{n}\subseteq(x_{1},...,x_{d})^{n}S$. Let $P_{n}$
be a prime ideal in $R$ lying over $\mathfrak{p}_{n}$. Then $P_{n}\ne0$
and $P_{n}\subseteq\mathfrak{m}^{n}$.
\end{proof}

\begin{prop}
\label{prop:reduction to dimension one} Let $(R,\mathfrak{m})$ be
an equicharacteristic Noetherian complete local ring. Let $I$ be
an ideal of $R$ such that $\text{dim}(R/P)>1$ for every associated
prime ideal $P$ of $R/I$. If there exists an element $r\in R$ such
that $r\in I^{(2)}\setminus\mathfrak{m}I$, then, there exists an
ideal $J$ such that $I\subsetneq J$, $r\in J^{(2)}\setminus\mathfrak{m}J$
and $\text{dim}(R/J)<\text{dim}(R/I)$. Moreover, if $I$ is height
unmixed $J$ can be chosen to be a height unmixed ideal. If $I$ is
radical, $J$ can be chosen to be radical. \end{prop}
\begin{proof}
Let $I=\mathfrak{p}_{1}\cap...\cap\mathfrak{p_{n}}$, where $\mathfrak{p_{i}}$
is a $P_{i}\text{-}$primary ideal, be the primary decomposition of
$I$. 

By hypothesis, $\text{dim}(R/P_{i})>1$. Then, by Proposition \ref{prop:complete local rings have small primes},
$R/P_{i}$ has non-zero prime ideals, say $\mathcal{Q}_{i,t}$ such
that $\mathcal{Q}_{i,t}\subseteq(\mathfrak{m}/P_{i})^{t}$ for all
positive integers $t$. Without loss of generality we may choose $\mathcal{Q}_{i,t}$
such that $\text{ht}(\mathcal{Q}_{i,t})=1$. Fix a positive integer
$t=\tau$ and let $Q_{i,\tau}$ denote the preimage of $\mathcal{Q}_{i,\tau}$
in $R$. Set $Q_{i}=Q_{i,\tau}$ for brevity of notation. Then $Q_{i}$
are prime ideals in $R$ such that $P_{i}\subsetneq Q_{i}\subseteq P_{i}+\mathfrak{m}^{\tau}$
for $1\leq i\leq n$. Then we define ideals 
\[
\mathfrak{q}_{i,m}:=(\mathfrak{p_{i}}+Q_{i}^{m}):(R\setminus Q_{i})=\{r\in R:rs\in(\mathfrak{p_{i}}+Q_{i}^{m})\text{ for some }s\in(R\setminus Q_{i})\}
\]
By definition, $\mathfrak{q}_{i,m}$ are $Q_{i}\text{-}$primary.

Next we claim that $\cap_{m\in\mathbb{Z}_{>0}}\mathfrak{q}_{i,m}=\mathfrak{p}_{i}$.
Let $r\in\cap_{m\in\mathbb{Z}_{>0}}\mathfrak{q}_{i,m}$. Then there
exists a $w_{m}\in R\setminus Q$ such that $rw_{m}\in(\mathfrak{p_{i}}+Q_{i}^{m})$
for all $m\in\mathbb{Z}_{>0}$. Thus, $r\in\cap_{m\in\mathbb{Z}_{>0}}(\mathfrak{p_{i}}+Q_{i}^{m})R_{Q_{i}}=\cap_{m\in\mathbb{Z}_{>0}}(\mathfrak{p_{i}}R_{Q_{i}}+Q_{i}^{m}R_{Q_{i}})$.
By Krull's intersection theorem applied to $R_{Q_{i}}$ we have that
$\cap_{m\in\mathbb{Z}_{>0}}(Q_{i}^{m}R_{Q_{i}})=0$. So $r\in\mathfrak{p_{i}}R_{Q_{i}}\cap R=\mathfrak{p}_{i}$,
where the last equality follows since $\mathfrak{p}_{i}$ is primary
(Proposition 4.8.ii, page 53, \citep{Atiyah-Macdonald}). 

Now, by Chevalley's theorem applied to $R/\mathfrak{p}_{i}$, there
exists a function $b_{i}:\mathbb{Z}_{>0}\to\mathbb{Z}_{>0}$ such
that $\mathfrak{q}_{i,b_{i}(N)}\subseteq\mathfrak{p}_{i}+\mathfrak{m}^{N}$
for all positive integers $N$. Let $J_{N}=\mathfrak{q}_{1,b_{1}(N)}\cap...\cap\mathfrak{q}_{n,b_{n}(N)}$.
Then $J_{N}\subseteq(\mathfrak{p}_{1}+\mathfrak{m}^{N})\cap...\cap(\mathfrak{p}_{n}+\mathfrak{m}^{N})$.
We claim that $J_{N}\subseteq I+\mathfrak{m}^{N-c}$ for $N\gg0$
and some positive integer $c<N$. We prove the claim by induction
on $n$. Suppose that $r\in(\mathfrak{p}_{1}+\mathfrak{m}^{N})\cap(\mathfrak{p}_{2}+\mathfrak{m}^{N})$.
Then we can write that $r=p_{1}+m_{1}=p_{2}+m_{2}$ for some $p_{i}\in\mathfrak{p}_{i}$
for $i=1,2$ and $m_{1},m_{2}\in\mathfrak{m}^{N}$. Then $m_{1}-m_{2}=p_{2}-p_{1}\in\mathfrak{m}^{N}\cap(\mathfrak{p}_{1}+\mathfrak{p}_{2})$.
By the Artin-Rees lemma, there exists a positive integer $c_{12}$
such that $\mathfrak{m}^{N}\cap(\mathfrak{p}_{1}+\mathfrak{p}_{2})=\mathfrak{m}^{N-c_{12}}(\mathfrak{m}^{c_{12}}\cap(\mathfrak{p}_{1}+\mathfrak{p}_{2}))\subseteq\mathfrak{m}^{N-c_{12}}(\mathfrak{p}_{1}+\mathfrak{p}_{2})=\mathfrak{m}^{N-c_{12}}\mathfrak{p}_{1}+\mathfrak{m}^{N-c_{12}}\mathfrak{p}_{2}$.
So $p_{2}-p_{1}\in\mathfrak{m}^{N-c_{12}}\mathfrak{p}_{1}+\mathfrak{m}^{N-c_{12}}\mathfrak{p}_{2}$.
Write $p_{2}-p_{1}=m_{1}^{'}+m_{2}^{'}$, where $m_{i}^{'}\in\mathfrak{m}^{N-c_{12}}\mathfrak{p}_{i}$
for $i=1,2$. Then $p_{1}+m_{1}^{'}=p_{2}-m_{2}^{'}$. Note that the
left hand side of this equation is an element of $\mathfrak{p}_{1}$
and the right hand side is an element of $\mathfrak{p}_{2}$. Thus,
each side is an element of $\mathfrak{p}_{1}\cap\mathfrak{p}_{2}$.
Now $r=p_{1}+m_{1}=(p_{1}+m_{1}^{'})-(m_{1}^{'}-m_{1})\in(\mathfrak{p}_{1}\cap\mathfrak{p}_{2})+\mathfrak{m}^{N-c_{12}}$.
Thus, $(\mathfrak{p}_{1}+\mathfrak{m}^{N})\cap(\mathfrak{p}_{2}+\mathfrak{m}^{N})\subseteq(\mathfrak{p}_{1}\cap\mathfrak{p}_{2})+\mathfrak{m}^{N-c_{12}}$
for $N\geq c_{12}$. Proceeding inductively, we can show that there
exists a positive integer $c$ such that for $N\geq c$, $(\mathfrak{p}_{1}+\mathfrak{m}^{N})\cap...\cap(\mathfrak{p}_{n}+\mathfrak{m}^{N})\subseteq(\mathfrak{p}_{1}\cap...\cap\mathfrak{p_{n}})+\mathfrak{m}^{N-c}$.
Consequently, $J_{N}\subseteq I+\mathfrak{m}^{N-c}$. Fix one such
$N\gg c$ and set $J=J_{N}$. 

By construction, $\text{dim}(R/J)<\text{dim}(R/I)$. Choose $u\in I^{(2)}$.
Then there exists $v\in R$ such that $v$ is not contained in any
minimal prime of $I$ and $uv\in I^{2}$. Since for all associated
primes $Q_{i}$ of $J$, $Q_{i}\subseteq P_{i}+\mathfrak{m}^{t}$,
by choosing $t\gg0$, we can ensure that $v$ is not contained in
any minimal prime of $J$. Then $vu\in J^{2}$ and $u\in J^{(2)}$.
Further, if $u\notin\mathfrak{m}I$, we claim that $u\notin\mathfrak{m}J$.
Suppose that $u\in\mathfrak{m}J$. Then, by the preceding paragraph,
$u\in\mathfrak{m}(I+\mathfrak{m}^{N-c})=\mathfrak{m}I+\mathfrak{m}^{N+1-c}$.
Write $u=v+w$, where $v\in\mathfrak{m}I$ and $w\in\mathfrak{m}^{N+1-c}$.
Since $u\in I^{(2)}\subseteq I$ and $v\in\mathfrak{m}I\subseteq I$,
$w\in I$. So $w\in\mathfrak{m}^{N+1-c}\cap I$. By the Artin-Rees
lemma there exists a positive integer $c^{'}$ such that for $N+1-c\geq c^{'}$,
$\mathfrak{m}^{N+1-c}\cap I=\mathfrak{m}^{N+1-c-c^{'}}(\mathfrak{m}^{c^{'}}\cap I)\subseteq\mathfrak{m}^{N+1-c-c^{'}}I\subseteq\mathfrak{m}I$.
Thus, for $N\gg0$, $u=v+w\in\mathfrak{m}I$, which contradicts the
choice of $u$. Hence, $u\in\mathfrak{m}J$. This proves the first
assertion in the proposition. 

If $I$ is height unmixed, then, all associated prime ideals of $I$
are minimal and have the same height. By choice of $Q_{i}$ all associated
prime ideals of $J$ will also have the same height, which is $1$
higher than the height of $I$. So $J$ is height unmixed. If $I$
is radical, we choose the primary decomposition of $I$ as an intersection
of it's minimal primes and choose the primary decomposition of $J$
as the intersection of the minimal $Q_{i}$. So $J$ is radical and
by the arguments in the preceding paragraph we obtain the desired
conclusion.
\end{proof}

We will now prove the last result referenced in Note \ref{special hypothesis}. 

Let $S=k[[t,x_{1},...,x_{m}]]$, where $k$ is a field and $P$ is
a prime ideal in $S$ such that $\text{dim}(S/P)=1$. Since $S/P$
is a one dimensional Noetherian complete local domain, its integral
closure, $\overline{S/P}$, is a one dimensional normal Noetherian
complete local domain (Theorem 2.2.5, page 31 and Theorem 4.3.4, page
62, \citep{Hunke-Swanson}) and hence regular (Theorem 14.1, page
198, \citep{kemper}). Thus, we can identify, $\overline{S/P}=k[[y]]$
for some indeterminate $y$ (Theorem 15, \citep{cohen}). Consequently,
we have an inclusion $F:S/P\hookrightarrow k[[y]]$. Under this inclusion,
let $t\mapsto y^{n}u$ and $x_{i}\mapsto g_{i}(y)$ for $1\leq i\leq m$,
where $n$ is a positive integer and $u$ is a unit in $k[[y]]$.
Write $u=u_{0}+f(y)$, where $f(y)$ is a power series in $y$ with
no constant term and $u_{0}\in k$. Suppose that $u_{0}$ has an $n$th
root in $k$ (in particular, this is true if $k$ is algebraically
closed and $n$ is invertible in $k$).
\begin{prop}
\label{prop: hypothesis is not special} With the notation as in the
preceding paragraph, let $R=k[[t^{\frac{1}{n}},x_{1},...,x_{m}]]$
and let $Q=(x_{1}-f_{1}(t^{\frac{1}{n}}),...,x_{m}-f_{m}(t^{\frac{1}{n}}))R$.
Then $P=Q\cap S$.\end{prop}
\begin{proof}
Now $u$ has an $n$th root in $k[[y]]$, say, $v^{n}=u$. Consider
the automorphism $V:k[[y]]\to k[[y]]$ given by $V(y)=yv^{-1}$. Then
we have an injective map $V\circ F:S/P\hookrightarrow k[[y]]$, where
$t\mapsto y^{n}$ and $x_{i}\mapsto g_{i}(yv^{-1})$. Set $f_{i}(y):=g_{i}(yv^{-1})$.
This induces a surjective map $G:R=k[[t^{\frac{1}{n}},x_{1},...,x_{m}]]\to k[[y]]$,
where $t^{\frac{1}{n}}\mapsto y$ and $x_{i}\mapsto f_{i}(y)$. Suppose
that the kernel of $G$ is $Q$. Note that $Q$ is prime in $R$ since
$k[[y]]$ is a domain. Further, since the restriction of $G$ to $S$
is the map sending $t\mapsto y^{n}$ and $x_{i}\mapsto f_{i}(y)$,
the kernel of $G|_{S}$ is $P$. Thus, $Q\cap S=P$. Finally, we have
that $Q=(x_{1}-f_{1}(t^{\frac{1}{n}}),...,x_{m}-f_{m}(t^{\frac{1}{n}}))R$
since this ideal is clearly in the kernel of $G$ by definition and
the quotient of $R$ modulo this ideal is precisely $k[[t]]$. This
proves our claim. 
\end{proof}
In particular, we have shown that if $S=k[[t,x_{1},...,x_{m}]]$,
where $k$ is a characteristic $0$ algebraically closed field and
$P$ is a prime in $S$ with $\text{dim}(S/P)=1$, then, there exists
a positive integer $n$ such that for $R=k[[t^{\frac{1}{n}},x_{1},...,x_{m}]]$,
there exists a prime ideal $Q=(x_{1}-f_{1}(t^{\frac{1}{n}}),...,x_{m}-f_{m}(t^{\frac{1}{n}}))R$
that contracts to $P$ in $S$.

\subsection{Problem set-up}

Let $R=k[[t,x_{1},...,x_{m}]]$ be the formal power series ring over
a field $k$ of characteristic $0$ in $m+1$ indeterminates. Then
the fraction field of $R$, say $K$, is $k((t,x_{1},...,x_{m}))$,
the ring of formal Laurent series over the same indeterminates. Let
$n>1$ be a positive integer and let $S=k[[t^{n},x_{1},...,x_{m}]]$,
where $n>1$ is a positive integer. The fraction field of $S$ can
be identified with $L=k((t^{n},x_{1},...,x_{m}))$. Note that that
the integral closure of $S$ in $L$ is $R$. 

Now assume that $k$ is algebraically closed if $n>2$ . Then $K/L$
is a Galois extension. Consider the automorphisms $\sigma_{j}$ of
$K$, where $\sigma_{j}(t)=\zeta^{j-1}t$ for $j=1,...,n$, where
$\zeta$ is a primitive $n$th root of unity and $\sigma_{j}(x_{i})=x_{i}$
for $1\leq i\leq m$. Any automorphism of $K$ that fixes $L$ must
fix each $x_{i}$ and must map an $n$th root of $t^{n}$ to another
$n$th root of $t^{n}$ and hence must map $t$ to $\zeta^{j-1}t$
for some $j\in\{1,...,n\}$. Thus, the Galois group $G$ of $K/L$
is $\{\sigma_{1},...,\sigma_{n}\}$. 

Consider any prime ideal $\mathfrak{p}$ in $S$. Suppose that $\mathfrak{Q}=\{\mathfrak{q}_{1},...,\mathfrak{q}_{l}\}$
is the set of prime ideals of $R$ lying over $\mathfrak{p}$. Then
$G$ acts transitively on $\mathfrak{Q}$ (Proposition VII.2.1, page
340, \citep{lang}). 

Now let $f_{1}(t),...,f_{m}(t)\in k[[t]]$ and let $Q_{1}=(x_{1}-f_{1}(t),...,x_{m}-f_{m}(t))R$.
Then $Q_{1}$ is a prime ideal as $R/Q_{1}=k[[t]]$, which is a domain.
If $P=Q_{1}\cap S$, the set of primes lying over $P$ are $Q_{j}=(x_{1}-f_{1}(\zeta^{j-1}t),...,x_{m}-f_{m}(\zeta^{j-1}t))R$
for $1\leq j\leq n$ since $G$ acts transitively on the set of primes
lying over $P$. It follows that $Q_{j}\cap S=P=(Q_{1}\cap S)\cap...\cap(Q_{n}\cap S)=(Q_{1}\cap...\cap Q_{n})\cap S$
for $1\leq j\leq n$.

To compute the generators of $P$, we will make use of the \emph{normalized
trace map} $\text{tr}:L\to K$, where for $a\in L$ we have that 
\[
\text{\text{tr}(a})=\frac{1}{n}(\sigma_{1}(a)+...+\sigma_{n}(a)).
\]

Note that by the above discussion, since $G$ acts transitively on
the set $\{Q_{1},...,Q_{n}\}$, $\sigma_{j}$ stabilizes $Q_{1}\cap...\cap Q_{n}$
for $j=1,...,n$ and thus, $\text{tr}(Q_{1}\cap...\cap Q_{n})\subseteq Q_{1}\cap...\cap Q_{n}$.
Further, since $\text{tr}(\cdot)|_{R}$ has range $S$, $\text{tr}(Q_{1}\cap...\cap Q_{n})\subseteq S$.
Consequently, we have that $ $$\text{tr}(Q_{1}\cap...\cap Q_{n})=(Q_{1}\cap...\cap Q_{n})\cap S=P$.

We are now ready to show that under the hypothesis of Theorem \ref{main theorem}
(which is a special case of the above set up for $n=2$), $P^{(2)}\subseteq(Q_{1}^{2}\cap Q_{2}^{2})\cap S$. 
\begin{prop}
\label{prop:containment of symbolic power in contraction} Let $k$
be a field (not necessarily of characteristic 0), $R=k[[t,x_{1},...,x_{m}]]$,
$S=k[[t^{n},x_{1},...,x_{m}]]$. Let $f_{1}(t),...,f_{m}(t)\in k[[t]]$.
If $n>2$, assume that $k$ is algebraically closed and let $Q_{j}=(x_{1}-f_{1}(\zeta^{j-1}t),...,x_{m}-f_{m}(\zeta^{j-1}t))R$,
where $\zeta$ is a primitive $n$th root of unity and $j\in\{1,...,n\}$.
Then $P^{(l)}\subseteq(Q_{1}^{l_{1}}\cap...\cap Q_{n}^{l_{n}})\cap S$,
where $P=(Q_{1}\cap...\cap Q_{n})\cap S$ and $l\geq\text{max}\{l_{1},...,l_{n}\}$.\end{prop}
\begin{proof}
The sequences $X_{j}=x_{1}-f_{1}(\zeta^{j-1}t),...,x_{m}-f_{m}(\zeta^{j-1}t)$
are regular sequences in $R$ for $1\leq j\leq n$, since $R/(x_{1}-f_{1}(\zeta^{j-1}t),...,x_{i}-f_{i}(\zeta^{j-1}t))R\cong k[[t,x_{i+1},...,x_{m}]]$.
Being a domain, the latter ring has no non-zero zero-divisors and
the class of $x_{i+1}-f_{i+1}(\zeta^{j-1}t)$ is not zero in this
ring for $0\leq i\leq m-1$. Also, by the same token, the ideals $Q_{j}$
are prime for $1\leq j\leq n$. Then we have that $Q_{j}^{(r)}=Q_{j}^{r}$
for every positive integer $r$ (result 2.1, \citep{Hochster}). Thus,
the ideals $Q_{j}^{r}$ are primary for every positive integer $r$
and $1\leq j\leq n$. Now the contraction of a primary ideal is primary
(Proposition 4.8, page 53, \citep{Atiyah-Macdonald}). Consequently,
the ideals $Q_{j}^{r}\cap S$ are primary in $S$.

Now $\sqrt{Q_{j}^{r}\cap S}=\sqrt{Q_{j}^{r}}\cap S=Q_{j}\cap S=P$
(exercise 1.13, page 9 and exercise 1.18 page 10, \citep{Atiyah-Macdonald}).
Thus, the ideals $Q_{j}^{l_{j}}\cap S$ are all $P$-primary. Hence
$(Q_{1}^{l_{1}}\cap...\cap Q_{n}^{l_{n}})\cap S$ is $P$-primary
(lemma 4.3, page 51, \citep{Atiyah-Macdonald}). Further, $P^{l}=((Q_{1}\cap...\cap Q_{n})\cap S)^{l}\subseteq(Q_{1}\cap...\cap Q_{n})^{l}\cap S\subseteq(Q_{1}^{l}\cap...\cap Q_{n}^{l})\cap S\subseteq(Q_{1}^{l_{1}}\cap...\cap Q_{n}^{l_{n}})\cap S$.
Let $\mathfrak{q}=(Q_{1}^{l_{1}}\cap...\cap Q_{n}^{l_{n}})\cap S$.

Finally, for any irredundant primary decomposition of $P^{l}$, the
$P$-primary ideal that must be used is $P^{(l)}$. Suppose that $P^{l}=P^{(l)}\cap P_{1}\cap...\cap P_{r}$
be an irredundant primary decomposition, where the $P_{1},...,P_{r}$
are primary ideals. Then $\sqrt{P_{1}},...,\sqrt{P_{r}}$ are all
distinct and are distinct from $P$. Further, $P^{(l)}\nsupseteq\cap_{i=1}^{r}P_{i}$
and $P_{i^{'}}\nsupseteq P^{(l)}\cap_{i=1,i\neq i^{'}}^{r}P_{i}$
for $1\leq i^{'}\leq r$. We claim that $P^{l}=(P^{(l)}\cap\mathfrak{q})\cap P_{1}\cap P_{2}\cap...\cap P_{r}$
is also an irredundant primary decomposition (note that $P^{l}=P^{l}\cap\mathfrak{q}$
as $P^{l}\subseteq\mathfrak{q}$ by the preceding paragraph). For
$P^{(l)}$ and $\mathfrak{q}$ are both $P$-primary and hence so
is $P^{(l)}\cap\mathfrak{q}$. So the radicals of all ideals appearing
in the decomposition are all distinct. Also, $P^{(l)}\cap\mathfrak{q}\nsupseteq\cap_{i=1}^{r}P_{i}$
follows from $P^{(l)}\nsupseteq\cap_{i=1}^{r}P_{i}$. Suppose that
$P_{i^{'}}\supseteq(P^{(l)}\cap\mathfrak{q})\cap_{i=1,i\neq i^{'}}^{r}P_{i}=(P^{(l)}\cap_{i=1,i\neq i^{'}}^{r}P_{i})\cap\mathfrak{q}$.
Then, since $P_{i^{'}}\nsupseteq P^{(l)}\cap_{i=1,i\neq i^{'}}^{r}P_{i}$,
we must have $P_{i^{'}}\supseteq\mathfrak{q}$ (Proposition 1.11.(ii),
page 8, \citep{Atiyah-Macdonald}). However, taking radicals, we get
that $\sqrt{\mathfrak{q}}=P\subseteq\sqrt{P_{i^{'}}}=P_{i^{'}}$,
which is a contradiction. Thus, $P_{i^{'}}\nsupseteq(P^{(l)}\cap\mathfrak{q})\cap_{i=1,i\neq i^{'}}^{r}P_{i}$.
So the new primary decomposition is indeed irredundant. Now, since
the $P$-primary component in any primary decomposition of $P^{l}$
must be $P^{(l)}$, we have that $P^{(l)}=P^{(l)}\cap\mathfrak{q}$.
So $P^{(l)}\subseteq\mathfrak{q}=(Q_{1}^{l_{1}}\cap...\cap Q_{n}^{l_{n}})\cap S$.\end{proof}
\begin{cor}
\label{cor:symbolic power containment} Let $k$ be a field (not necessarily
of characteristic 0, $R=k[[t,x_{1},...,x_{m}]]$, $S=k[[t^{2},x_{1},...,x_{m}]]$.
Let $f_{1}(t),...,f_{m}(t)\in k[[t]]$. Let $Q_{1}=(x_{1}-f_{1}(t),...,x_{m}-f_{m}(t))R$
and \textup{$Q_{2}=(x_{1}-f_{1}(-t),...,x_{m}-f_{m}(-t))R$}. Then
$P^{(2)}\subseteq(Q_{1}^{2}\cap Q_{2}^{2})\cap S$, where $P=(Q_{1}\cap Q_{2})\cap S$.\end{cor}
\begin{proof}
This is a direct consequence of Proposition \ref{prop:containment of symbolic power in contraction}
with $n=2$, $l_{0}=l_{1}=l=2$.
\end{proof}

We will prove the stronger containment $(Q_{1}^{2}\cap Q_{2}^{2})\cap S\subseteq\mathfrak{m}P$
in the next sections, which will imply the Eisenbud-Mazur conjecture
in this case by corollary \ref{cor:symbolic power containment}.

\subsection{\label{section with notation change} Computing generators of $P$}

Let the notation be as in the statement of Theorem \ref{main theorem}.
We may assume that at least one of the power series $f_{1}(t),...,f_{m}(t)$
is not in $S$. Otherwise, $f_{i}(t)=f_{i}(-t)$ for $i=1,...,m$
(since $t^{2}\in S$) and 
\[
Q_{2}=(x_{1}-f_{1}(-t),...,x_{m}-f_{m}(-t))R=(x_{1}-f_{1}(t),...,x_{m}-f_{m}(t))R=Q_{1}.
\]
Thus, $P=Q_{1}\cap S=(x_{1}-f_{1}(t),...,x_{m}-f_{m}(t))S$ is a complete
intersection. Then, $P^{(2)}=P^{2}\subseteq\mathfrak{m}P$ (Result
2.1, \citep{Hochster}).

Without loss of generality, we may assume that $(f_{1}(t)-f_{1}(-t))|(f_{i}(t)-f_{i}(-t))$
and set $g_{i}(t)=\frac{(f_{i}(t)-f_{i}(-t))}{(f_{1}(t)-f_{1}(-t))}$
for $i=1,...,m$ (else we may renumber so that the leading term of
$f_{1}(t)-f_{1}(-t)$ has the least non-zero degree). Set $a_{i}=x_{i}-f_{i}(t)$
and $b_{i}=x_{i}-f_{i}(-t)$ for $i=1,...,m$.
\begin{prop}
\label{Q1 meet Q2} With the notation as in the preceding paragraph,
$P=(Q_{1}\cap Q_{2})\cap S=(\{(-b_{i}+b_{1}g_{i}(t)):i=2,...,m\})S+(\{\text{tr}((x_{i}-f_{i}(t))(x_{j}-f_{j}(-t))):i,j\in\{1,...,m\}\})S.$ \end{prop}
\begin{proof}
We have that $b_{i}-a_{i}=f_{i}(t)-f_{i}(-t)=(f_{1}(t)-f_{1}(-t))g_{i}(t)=(b_{1}-a_{1})g_{i}(t)$
for $i=1,...,m$.

Let $u\in Q_{1}\cap Q_{2}$. Then, for some $r_{i},s_{i}\in R$, we
may write that $u=\Sigma_{i=1}^{m}r_{i}a_{i}=\Sigma_{i=1}^{m}s_{i}b_{i}=\Sigma_{i=1}^{m}s_{i}(a_{i}+(b_{i}-a_{i}))=\Sigma_{i=1}^{m}s_{i}(a_{i}+(b_{1}-a_{1})g_{i}(t))$.
Hence, $\Sigma_{i=1}^{m}(r_{i}-s_{i})a_{i}=(b_{1}-a_{1})\Sigma_{i=1}^{m}s_{i}g_{i}(t)$.
Note the left hand side of the last equation lies in $Q_{1}$.

Thus, elements of $Q_{1}\cap Q_{2}$ are determined by elements $s_{i}\in R$
such that $\Sigma_{i=1}^{n}s_{i}g_{i}(t)\in Q_{1}:(b_{1}-a_{1})$,
as given any $s_{i}$ satisfying this condition, we may determine
the $r_{i}$ from the preceding equation, thus obtaining an element
of $Q_{1}\cap Q_{2}$.

Since $Q_{1}$ is a prime and $(b_{1}-a_{1})\notin Q_{1}$, we have
that $Q_{1}:(b_{1}-a_{1})=Q_{1}$. So elements of $Q_{1}\cap Q_{2}$
are determined by elements $s_{i}\in R$ such that $\Sigma_{i=1}^{m}s_{i}g_{i}(t)\in Q_{1}$.
Modulo $Q_{1}$, these are the preimages of the elements defining
the relations between $g_{i}(t)$ in $R/Q_{1}=k[[t]]$. Thus, $Q_{1}\cap Q_{2}$
is generated by elements $\Sigma_{i=1}^{m}s_{i}b_{i}$, where $\Sigma_{i=1}^{m}s_{i}g_{i}(t)$
represents the zero element in $R/Q_{1}$. 

Given elements $w_{1},...,w_{d}\in k[[t]]$ we define a \emph{relation}
among these elements to be a $d\text{-}$tuple $(\alpha_{1},...,\alpha_{d})\in k[[t]]^{d}$
such that $\alpha_{1}w_{1}+...+\alpha_{d}w_{d}=0$. The set of such
elements is a submodule of $k[[t]]^{n}$, which we shall call the
\emph{module of relations. }Note that since $k[[t]]$ is a principal
ideal domain, the module of relations is a free $k[[t]]\text{-}$module.
Now the module of relations between the $g_{i}(t)$ in $R/Q_{1}=k[[t]]$
is generated by the following $m\text{-}$tuples (note that $g_{1}(t)=1$):
\[
(g_{i}(t),0,0,...,\underbrace{-1}_{i^{\text{th }}\text{place}},0,0,...,0)\text{\text{ (where \ensuremath{i=2,....,m})\text{}}}.
\]

Let, $h_{i}=-b_{i}+b_{1}g_{i}(t)=-a_{i}+a_{1}g_{i}(t)\in Q_{1}\cap Q_{2}$
(where the equality follows from $b_{i}-a_{i}=(b_{1}-a_{1})g_{i}(t)$).
The set of $m\text{-}$tuples above correspond to the elements $g_{i}(t)b_{1}+0b_{2}+0b_{3}+...+(-1)b_{i}+0b_{i+1}+0b_{i+2}+...+0b_{m}=b_{1}g_{i}(t)-b_{i}=h_{i}$
in $Q_{1}\cap Q_{2}$. Thus, $Q_{1}\cap Q_{2}=(h_{2},h_{3},...,h_{m})R+Q_{1}Q_{2}$. 

Now we compute $(Q_{1}\cap Q_{2})\cap S=\text{tr}(Q_{1}\cap Q_{2})$.
We have that $\text{tr}(h_{i})=\text{tr}(-b_{i}+b_{1}g_{i}(t))=\text{tr}(-x_{i}+f_{i}(-t)+(x_{1}-f_{1}(-t))g_{i}(t))$.
Now $g_{i}(-t)=\frac{(f_{i}(-t)-f_{i}(t))}{(f_{1}(-t)-f_{1}(t))}=\frac{(f_{i}(t)-f_{i}(-t))}{(f_{1}(t)-f_{1}(-t))}=g_{i}(t)$.
Hence, 
\begin{eqnarray*}
\text{tr}(h_{i}) & = & \frac{1}{2}((-x_{i}+f_{i}(-t)+(x_{1}-f_{1}(-t))g_{i}(t)))\\
 &  & +\frac{1}{2}((-x_{i}+f_{i}(t)+(x_{1}-f_{1}(t))g_{i}(t)))\\
 & = & \frac{1}{2}((-b_{i}+b_{1}g_{i}(t))+(-a_{i}+a_{1}g_{i}(t)))\\
 & = & \frac{1}{2}(2(-b_{i}+b_{1}g_{i}(t)))\\
 & = & -b_{i}+b_{1}g_{i}(t)\\
 & = & h_{i}
\end{eqnarray*}

Next, we have that $Q_{1}Q_{2}=(\{(x_{i}-f_{i}(t))(x_{j}-f_{j}(-t)):i,j\in\{1,...,m\}\})R$.
Therefore, $(P=(h_{2},h_{3},...,h_{m})S+(\{\text{tr}((x_{i}-f_{i}(t))(x_{j}-f_{j}(-t))):i,j\in\{1,...,m\}\})S$.
\end{proof}

\subsection{\label{special case q12 cap q22}Computing generators of $Q_{1}^{2}\cap Q_{2}^{2}$
(special case)}

Again, assume the hypothesis of Theorem \ref{main theorem}. For the
purpose of computing generators of $Q_{1}^{2}\cap Q_{2}^{2}$ we show
that it is sufficient to consider power series $f_{1}(t),...,f_{m}(t)$
containing only odd powers of $t$, \emph{i.e.,} those $f_{i}(t)$
that satisfy $f_{i}(-t)=-f_{i}(t)$ for $1\leq i\leq m$. Suppose
that $f_{i}(t)=\Sigma_{h=0}^{\infty}a_{i,h}t^{h}$. Then we may write
that $ $$f_{i}(t)=\Sigma_{h=0}^{\infty}a_{i,2h}t^{2h}+\Sigma_{h=0}^{\infty}a_{i,2h+1}t^{2h+1}$.
Set $f_{i,e}=\Sigma_{h=0}^{\infty}a_{i,2h}t^{2h}$ and $f_{i,o}=\Sigma_{h=0}^{\infty}a_{i,2h+1}t^{2h+1}$
for the even and odd parts of $f_{i}$ respectively. Consider the
automorphism $\sigma$ of $R$: $\sigma(x_{i})=x_{i}-f_{i,e}(t)$
and $\sigma(t)=t$. We have that $Q_{1}=(x_{1}-f_{1}(t),...,x_{m}-f_{m}(t))R=(x_{1}-f_{1,e}(t)-f_{1,o}(t),...,x_{m}-f_{m,e}(t)-f_{m,o}(t))R$.
So that $\sigma(Q_{1})=(x_{1}-f_{1,o}(t),x_{2}-f_{2,o}(t),...,x_{m}-f_{m,o}(t))R$.
Similarly, we have that $Q_{2}=(x_{1}-f_{1}(-t),...,x_{m}-f_{m}(-t))R=(x_{1}-f_{1,e}(t)+f_{1,o}(t),...,x_{m}-f_{m,e}(t)+f_{m,o}(t))R$.
So that $\sigma(Q_{2})=(x_{1}+f_{1,o}(t),x_{2}+f_{2,o}(t),...,x_{m}+f_{m,o}(t))R$.
So the problem reduces to the case where $f_{i}(t)$ are odd power
series. 

We rewrite the result of Proposition \ref{Q1 meet Q2} under this
reduction. We have that $g_{i}(t)=\frac{(f_{i}(t)-f_{i}(-t))}{(f_{1}(t)-f_{1}(-t))}=\frac{f_{i}(t)}{f_{1}(t)}$.
Also, $-b_{i}+b_{1}g_{i}(t)=-(x_{i}+f_{i}(t))+(x_{1}+f_{1}(t))g_{i}(t)=-x_{i}-f_{i}(t)+x_{1}g_{i}(t)+f_{i}(t)=-x_{i}+x_{1}g_{i}(t)$.
Thus, $P=(Q_{1}\cap Q_{2})\cap S=(\{(-b_{i}+b_{1}g_{i}(t)):i=2,...,m\})S+(\{\text{tr}((x_{i}-f_{i}(t))(x_{j}-f_{j}(-t))):i,j\in\{1,...,m\}\})S=(\{(-x_{i}+x_{1}g_{i}(t)):i=2,...,m\})S+(\{(x_{i}x_{j}-f_{i}(t)f_{j}(t)):i,j\in\{1,...,m\}\})S$.

For the sake of notational sanity we first illustrate the method for
the case when $m=3$ and discuss the generalization after that. 

We have that $Q_{1}=(x_{1}-f_{1}(t),x_{2}-f_{2}(t),x_{3}-f_{3}(t))R$
and $Q_{2}=(x_{1}+f_{1}(t),x_{2}+f_{2}(t),x_{3}+f_{3}(t))R$. Denote
$f_{i}(t)=f_{i}$ , $a_{i}=x_{i}-f_{i}$ and $b_{i}=x_{i}+f_{i}$
for $i=1,2,3$. Then $Q_{1}^{2}=(a_{1}^{2},a_{2}^{2},a_{3}^{2},a_{2}a_{3},a_{1}a_{3},a_{1}a_{2})R$
and $Q_{2}^{2}=(b_{1}^{2},b_{2}^{2},b_{3}^{2},b_{2}b_{3},b_{1}b_{3},b_{1}b_{2})R$.
If $u\in Q_{1}^{2}\cap Q_{2}^{2}$, then, we can write that $u=r_{1}a_{1}^{2}+r_{2}a_{2}^{2}+r_{3}a_{3}^{2}+s_{1}a_{2}a_{3}+s_{2}a_{1}a_{3}+s_{3}a_{1}a_{2}$
for some $r_{h},s_{h}\in R$, $h\in\{1,2,3\}$. Similarly, $u=r_{1}^{'}b_{1}^{2}+r_{2}^{'}b_{2}^{2}+r_{3}^{'}b_{3}^{2}+s_{1}^{'}b_{2}b_{3}+s_{2}^{'}b_{1}b_{3}+s_{3}^{'}b_{1}b_{2}$
for some $r_{h}^{'},s_{h}^{'}\in R$, $h\in\{1,2,3\}$. Now $b_{h}^{2}-a_{h}^{2}=4x_{i}f_{i}$
and $b_{h}b_{h^{'}}-a_{h}a_{h^{'}}=2x_{h}f_{h}+2x_{h^{'}}f_{h^{'}}$
with $h,h^{'}\in\{1,2,3\}$ and $h\neq h^{'}$. Equating the two expressions
for $u$, we have that $(r_{1}-r_{1}^{'})a_{1}^{2}+(r_{2}-r_{2}^{'})a_{2}^{2}+(r_{3}-r_{3}^{'})a_{3}^{2}+(s_{1}-s_{1}^{'})a_{2}a_{3}+(s_{2}-s_{2}^{'})a_{1}a_{3}+(s_{3}-s_{3}^{'})a_{1}a_{2}=r_{1}^{'}(4x_{1}f_{1})+r_{2}^{'}(4x_{2}f_{2})+r_{3}^{'}(4x_{3}f_{3})+s_{1}^{'}(2x_{2}f_{3}+2x_{3}f_{2})+s_{2}^{'}(2x_{1}f_{3}+2x_{3}f_{1})+s_{3}^{'}(2x_{1}f_{2}+2x_{2}f_{1})$.
The left hand side of this equation lies in $Q_{1}^{2}$, hence so
does the right hand side. The set of coefficients $\{r_{1}^{'},r_{2}^{'},r_{3}^{'},s_{1}^{'},s_{2}^{'},s_{3}^{'}\}$,
which determine elements of $Q_{1}^{2}\cap Q_{2}^{2}$ are completely
determined by the relations on $\{4x_{1}f_{1},4x_{2}f_{2},4x_{3}f_{3},2x_{2}f_{3}+2x_{3}f_{2},2x_{1}f_{3}+2x_{3}f_{1},2x_{1}f_{2}+2x_{2}f_{1}\}$
over $R/Q_{1}^{2}$ along with any $6$-tuple of elements in $Q_{1}^{2}$.
We now proceed to find the relations on $\{4x_{1}f_{1},4x_{2}f_{2},4x_{3}f_{3},2x_{2}f_{3}+2x_{3}f_{2},2x_{1}f_{3}+2x_{3}f_{1},2x_{1}f_{2}+2x_{2}f_{1}\}$
over $R/Q_{1}^{2}$. 

In $R/Q_{1}^{2}$ we have the following equations 
\begin{equation}
x_{1}^{2}=2x_{1}f_{1}-f_{1}^{2},x_{2}^{2}=2x_{2}f_{2}-f_{2}^{2},x_{3}^{2}=2x_{3}f_{3}-f_{3}^{2}\label{eq:relations q1 squared 1}
\end{equation}
\begin{equation}
x_{1}x_{2}=x_{1}f_{2}+x_{2}f_{1}-f_{1}f_{2},x_{2}x_{3}=x_{2}f_{3}+x_{3}f_{2}-f_{2}f_{3},x_{1}x_{3}=x_{1}f_{3}+x_{3}f_{1}-f_{1}f_{3}\label{eq:relations q1 squared 2}
\end{equation}

Thus, any element of $R/Q_{1}^{2}$ can be represented as $F_{0}(t)+F_{1}(t)x_{1}+F_{2}(t)x_{2}+F_{3}(t)x_{3}$,
where $F_{i}(t)\in k[[t]]$ for $i=0,1,2,3$. Thus, there is a one-one
correspondence between the relations on $\{4x_{1}f_{1},4x_{2}f_{2},4x_{3}f_{3},2x_{2}f_{3}+2x_{3}f_{2},2x_{1}f_{3}+2x_{3}f_{1},2x_{1}f_{2}+2x_{2}f_{1}\}$
over $R/Q_{1}^{2}$ and the relations between 
\[
E=\{4x_{1}f_{1},(4x_{1}f_{1})x_{1},(4x_{1}f_{1})x_{2},(4x_{1}f_{1})x_{3},
\]
\[
4x_{2}f_{2},(4x_{2}f_{2})x_{1},(4x_{2}f_{2})x_{2},(4x_{2}f_{2})x_{3},
\]
\[
4x_{3}f_{3},(4x_{3}f_{3})x_{1},(4x_{3}f_{3})x_{2},(4x_{3}f_{3})x_{3},
\]
\[
2x_{2}f_{3}+2x_{3}f_{2},(2x_{2}f_{3}+2x_{3}f_{2})x_{1},(2x_{2}f_{3}+2x_{3}f_{2})x_{2},(2x_{2}f_{3}+2x_{3}f_{2})x_{3},
\]
\[
2x_{1}f_{3}+2x_{3}f_{1},(2x_{1}f_{3}+2x_{3}f_{1})x_{1},(2x_{1}f_{3}+2x_{3}f_{1})x_{2},(2x_{1}f_{3}+2x_{3}f_{1})x_{3},
\]
\[
2x_{1}f_{2}+2x_{2}f_{1},(2x_{1}f_{2}+2x_{2}f_{1})x_{1},(2x_{1}f_{2}+2x_{2}f_{1})x_{2},(2x_{1}f_{2}+2x_{2}f_{1})x_{3}\}
\]
over $k[[t]]$. We rewrite these elements as $k[[t]]\text{-}$linear
combinations of $x_{1},x_{2},x_{3}$ and represent the coefficients
of $x_{1},x_{2},x_{3}$ and the term independent of these in a matrix
as follows. We abuse notation and denote the equivalence classes of
elements in $R$ modulo $Q_{1}^{2}$ by the same symbols as the elements
themselves.
\begin{enumerate}
\item $4x_{1}f_{1}$

\begin{enumerate}
\item $4x_{1}f_{1}=0+4f_{1}x_{1}+0x_{2}+0x_{3}$.
\item $(4x_{1}f_{1})x_{1}=4f_{1}x_{1}^{2}=4f_{1}(2x_{1}f_{1}-f_{1}^{2})=-4f_{1}^{3}+8f_{1}^{2}x_{1}+0x_{2}+0x_{3}$.
\item $(4x_{1}f_{1})x_{2}=4f_{1}x_{1}x_{2}=4f_{1}(x_{1}f_{2}+x_{2}f_{1}-f_{1}f_{2})=-4f_{1}^{2}f_{2}+4f_{1}f_{2}x_{1}+4f_{1}^{2}x_{2}+0x_{3}$.
\item $(4x_{1}f_{1})x_{3}=4f_{1}x_{1}x_{3}=4f_{1}(x_{1}f_{3}+x_{3}f_{1}-f_{1}f_{3})=-4f_{1}^{2}f_{3}+4f_{1}f_{3}x_{1}+0x_{2}+4f_{1}^{2}x_{3}$.
\end{enumerate}

We represent this data in the following matrix:

\begin{tabular}{l|cccc|}
\multicolumn{1}{l}{\multirow{1}{*}{coefficients in}} & $4x_{1}f_{1}$ & $(4x_{1}f_{1})x_{1}$ & $(4x_{1}f_{1})x_{2}$ & \multicolumn{1}{c}{$(4x_{1}f_{1})x_{3}$}\tabularnewline
term independent of $x_{1},x_{2},x_{3}$ & $0$ & $-4f_{1}^{3}$ & $-4f_{1}^{2}f_{2}$ & $-4f_{1}^{2}f_{3}$\tabularnewline
coefficient of $x_{1}$ & $4f_{1}$ & $8f_{1}^{2}$ & $4f_{1}f_{2}$ & $4f_{1}f_{3}$\tabularnewline
coefficient of $x_{2}$ & $0$ & $0$ & $4f_{1}^{2}$ & $0$\tabularnewline
coefficient of $x_{3}$ & $0$ & $0$ & $0$ & $ $$4f_{1}^{2}$\tabularnewline
\end{tabular}\medskip{}

We capture the above data in matrix $M_{1}$ below, 
\[
M_{1}=\left[\begin{array}{cccc}
0 & -4f_{1}^{3} & -4f_{1}^{2}f_{2} & -4f_{1}^{2}f_{3}\\
4f_{1} & 8f_{1}^{2} & 4f_{1}f_{2} & 4f_{1}f_{3}\\
0 & 0 & 4f_{1}^{2} & 0\\
0 & 0 & 0 & 4f_{1}^{2}
\end{array}\right].
\]

\item $4x_{2}f_{2}$: We represent elements $4x_{2}f_{2},(4x_{2}f_{2})x_{1},(4x_{2}f_{2})x_{2},(4x_{2}f_{2})x_{3}$
as $k[[t]]$-linear combinations of $x_{1},x_{2},x_{3}$ and collect
the coefficients in a matrix form in an analogous fashion. The associated
matrix is 
\[
M_{2}=\left[\begin{array}{cccc}
0 & -4f_{1}f_{2}^{2} & -4f_{2}^{3} & -4f_{2}^{2}f_{3}\\
0 & 4f_{2}^{2} & 0 & 0\\
4f_{2} & 4f_{1}f_{2} & 8f_{2}^{2} & 4f_{2}f_{3}\\
0 & 0 & 0 & 4f_{2}^{2}f_{3}
\end{array}\right].
\]

\item $4x_{3}f_{3}$: We repeat the above process for $4x_{3}f_{3},(4x_{3}f_{3})x_{1},(4x_{3}f_{3})x_{2},(4x_{3}f_{3})x_{3}$
and the associated matrix is 
\[
M_{3}=\left[\begin{array}{cccc}
0 & -4f_{1}f_{3}^{2} & -4f_{2}f_{3}^{2} & -4f_{3}^{3}\\
0 & 4f_{3}^{2} & 0 & 0\\
0 & 0 & 4f_{3}^{2} & 0\\
4f_{3} & 4f_{1}f_{3} & 4f_{2}f_{3} & 8f_{3}^{2}
\end{array}\right].
\]

\item $2x_{2}f_{3}+2x_{3}f_{2}$:

\begin{enumerate}
\item $2x_{2}f_{3}+2x_{3}f_{2}=0+0x_{1}+2f_{3}x_{2}+2f_{2}x_{3}$.
\item $(2x_{2}f_{3}+2x_{3}f_{2})x_{1}=2f_{3}x_{1}x_{2}+2f_{2}x_{1}x_{3}=2f_{3}(x_{1}f_{2}+x_{2}f_{1}-f_{1}f_{2})+2f_{2}(x_{1}f_{3}+x_{3}f_{1}-f_{1}f_{3})=-4f_{1}f_{2}f_{3}+4f_{2}f_{3}x_{1}+2f_{1}f_{3}x_{2}+2f_{1}f_{2}x_{3}$.
\item $(2x_{2}f_{3}+2x_{3}f_{2})x_{2}=2f_{3}x_{2}^{2}+2f_{2}x_{2}x_{3}=2f_{3}(2x_{2}f_{2}-f_{2}^{2})+2f_{2}(x_{2}f_{3}+x_{3}f_{2}-f_{2}f_{3})=-4f_{2}^{2}f_{3}+0x_{1}+6f_{2}f_{3}x_{2}+2f_{2}^{2}x_{3}$.
\item $(2x_{2}f_{3}+2x_{3}f_{2})x_{3}=2f_{3}x_{2}x_{3}+2f_{2}x_{3}^{2}=2f_{3}(x_{2}f_{3}+x_{3}f_{2}-f_{2}f_{3})+2f_{2}(2x_{3}f_{3}-f_{3}^{2})=-4f_{2}f_{3}^{2}+0x_{1}+2f_{3}^{2}x_{2}+6f_{2}f_{3}x_{3}$.
\end{enumerate}

So the associated matrix is 
\[
M_{23}=\left[\begin{array}{cccc}
0 & -4f_{1}f_{2}f_{3} & -4f_{2}^{2}f_{3} & -4f_{2}f_{3}^{2}\\
0 & 4f_{2}f_{3} & 0 & 0\\
2f_{3} & 2f_{1}f_{3} & 6f_{2}f_{3} & 2f_{3}^{2}\\
2f_{2} & 2f_{1}f_{2} & 2f_{2}^{2} & 6f_{2}f_{3}
\end{array}\right].
\]

\item Working as in the preceding case, the associated matrix for $2x_{1}f_{3}+2x_{3}f_{1},(2x_{1}f_{3}+2x_{3}f_{1})x_{1},(2x_{1}f_{3}+2x_{3}f_{1})x_{2},(2x_{1}f_{3}+2x_{3}f_{1})x_{3}$
is 
\[
M_{13}=\left[\begin{array}{cccc}
0 & -4f_{1}^{2}f_{3} & -4f_{1}f_{2}f_{3} & -4f_{1}f_{3}^{2}\\
2f_{3} & 6f_{1}f_{3} & 2f_{2}f_{3} & 2f_{3}^{2}\\
0 & 0 & 4f_{1}f_{3} & 0\\
2f_{1} & 2f_{1}^{2} & 2f_{1}f_{2} & 6f_{1}f_{3}
\end{array}\right].
\]

\item Finally, the associated matrix for $2x_{1}f_{2}+2x_{2}f_{1},(2x_{1}f_{2}+2x_{2}f_{1})x_{1},(2x_{1}f_{2}+2x_{2}f_{1})x_{2},(2x_{1}f_{2}+2x_{2}f_{1})x_{3}$
is 
\[
M_{12}=\left[\begin{array}{cccc}
0 & -4f_{1}^{2}f_{2} & -4f_{1}f_{2}^{2} & -4f_{1}f_{2}f_{3}\\
2f_{2} & 6f_{1}f_{2} & 2f_{2}^{2} & 2f_{2}f_{3}\\
2f_{1} & 2f_{1}^{2} & 6f_{1}f_{2} & 2f_{1}f_{3}\\
0 & 0 & 0 & 4f_{1}f_{2}
\end{array}\right].
\]

\end{enumerate}
We now consider the matrix $M=[M_{1}|M_{2}|M_{3}|M_{23}|M_{13}|M_{23}]$. 

\begin{center}
\arraycolsep=1pt

{\footnotesize 
\[
M=\left[\begin{array}{cccccccccccc}
0 & -4f_{1}^{3} & -4f_{1}^{2}f_{2} & -4f_{1}^{2}f_{3} & 0 & -4f_{1}f_{2}^{2} & -4f_{2}^{3} & -4f_{2}^{2}f_{3} & 0 & -4f_{1}f_{3}^{2} & -4f_{2}f_{3}^{2} & -4f_{3}^{3}\\
4f_{1} & 8f_{1}^{2} & 4f_{1}f_{2} & 4f_{1}f_{3} & 0 & 4f_{2}^{2} & 0 & 0 & 0 & 4f_{3}^{2} & 0 & 0\\
0 & 0 & 4f_{1}^{2} & 0 & 4f_{2} & 4f_{1}f_{2} & 8f_{2}^{2} & 4f_{2}f_{3} & 0 & 0 & 4f_{3}^{2} & 0\\
0 & 0 & 0 & 4f_{1}^{2} & 0 & 0 & 0 & 4f_{2}^{2} & 4f_{3} & 4f_{1}f_{3} & 4f_{2}f_{3} & 8f_{3}^{2}
\end{array}\right.
\]
}{\footnotesize \par}

{\footnotesize 
\[
\left.\begin{array}{cccccccccccc}
0 & -4f_{1}f_{2}f_{3} & -4f_{2}^{2}f_{3} & -4f_{2}f_{3}^{2} & 0 & -4f_{1}^{2}f_{3} & -4f_{1}f_{2}f_{3} & -4f_{1}f_{3}^{2} & 0 & -4f_{1}^{2}f_{2} & -4f_{1}f_{2}^{2} & -4f_{1}f_{2}f_{3}\\
0 & 4f_{2}f_{3} & 0 & 0 & 2f_{3} & 6f_{1}f_{3} & 2f_{2}f_{3} & 2f_{3}^{2} & 2f_{2} & 6f_{1}f_{2} & 2f_{2}^{2} & 2f_{2}f_{3}\\
2f_{3} & 2f_{1}f_{3} & 6f_{2}f_{3} & 2f_{3}^{2} & 0 & 0 & 4f_{1}f_{3} & 0 & 2f_{1} & 2f_{1}^{2} & 6f_{1}f_{2} & 2f_{1}f_{3}\\
2f_{2} & 2f_{1}f_{2} & 2f_{2}^{2} & 6f_{2}f_{3} & 2f_{1} & 2f_{1}^{2} & 2f_{1}f_{2} & 6f_{1}f_{3} & 0 & 0 & 0 & 4f_{1}f_{2}
\end{array}\right].
\]
}{\footnotesize \par}

\end{center}

Each column of $M$ encodes the $k[[t]]\text{-}$coefficients of the
elements of $E$. We will compute the $k[[t]]\text{-}$relations on
the columns of matrix $M$ and these will correspond to the $k[[t]]$-relations
on the elements of $E$. From these we can recover the relations $\{4x_{1}f_{1},4x_{2}f_{2},4x_{3}f_{3},2x_{2}f_{3}+2x_{3}f_{2},2x_{1}f_{3}+2x_{3}f_{1},2x_{1}f_{2}+2x_{2}f_{1}\}$
over $R/Q_{1}^{2}$ .

Multiplying the first row of $M$ by $-1$, we get the matrix 

\begin{center}
\arraycolsep=3pt

{\footnotesize 
\[
\left[\begin{array}{cccccccccccc}
0 & 4f_{1}^{3} & 4f_{1}^{2}f_{2} & 4f_{1}^{2}f_{3} & 0 & 4f_{1}f_{2}^{2} & 4f_{2}^{3} & 4f_{2}^{2}f_{3} & 0 & 4f_{1}f_{3}^{2} & 4f_{2}f_{3}^{2} & 4f_{3}^{3}\\
4f_{1} & 8f_{1}^{2} & 4f_{1}f_{2} & 4f_{1}f_{3} & 0 & 4f_{2}^{2} & 0 & 0 & 0 & 4f_{3}^{2} & 0 & 0\\
0 & 0 & 4f_{1}^{2} & 0 & 4f_{2} & 4f_{1}f_{2} & 8f_{2}^{2} & 4f_{2}f_{3} & 0 & 0 & 4f_{3}^{2} & 0\\
0 & 0 & 0 & 4f_{1}^{2} & 0 & 0 & 0 & 4f_{2}^{2} & 4f_{3} & 4f_{1}f_{3} & 4f_{2}f_{3} & 8f_{3}^{2}
\end{array}\right.
\]
\[
\left.\begin{array}{cccccccccccc}
0 & 4f_{1}f_{2}f_{3} & 4f_{2}^{2}f_{3} & 4f_{2}f_{3}^{2} & 0 & 4f_{1}^{2}f_{3} & 4f_{1}f_{2}f_{3} & 4f_{1}f_{3}^{2} & 0 & 4f_{1}^{2}f_{2} & 4f_{1}f_{2}^{2} & 4f_{1}f_{2}f_{3}\\
0 & 4f_{2}f_{3} & 0 & 0 & 2f_{3} & 6f_{1}f_{3} & 2f_{2}f_{3} & 2f_{3}^{2} & 2f_{2} & 6f_{1}f_{2} & 2f_{2}^{2} & 2f_{2}f_{3}\\
2f_{3} & 2f_{1}f_{3} & 6f_{2}f_{3} & 2f_{3}^{2} & 0 & 0 & 4f_{1}f_{3} & 0 & 2f_{1} & 2f_{1}^{2} & 6f_{1}f_{2} & 2f_{1}f_{3}\\
2f_{2} & 2f_{1}f_{2} & 2f_{2}^{2} & 6f_{2}f_{3} & 2f_{1} & 2f_{1}^{2} & 2f_{1}f_{2} & 6f_{1}f_{3} & 0 & 0 & 0 & 4f_{1}f_{2}
\end{array}\right].
\]
}\end{center}

Using notation introduced earlier, we write that $f_{2}=f_{1}g_{2}$
and $f_{3}=f_{1}g_{3}$. Denote the $i$th column of the matrix under
consideration by $C_{i}$. We perform the following column operations
on the preceding matrix:

$C_{2}-2f_{1}C_{1}$, $C_{3}-f_{2}C_{1}$, $C_{4}-f_{3}C_{1}$, $C_{6}-f_{1}g_{2}^{2}C_{1}$,
$C_{10}-f_{1}g_{3}^{2}C_{1}$, $C_{14}-f_{1}g_{2}g_{3}C_{1}$, $C_{17}-\frac{1}{2}g_{3}C_{1}$,
$C_{18}-\frac{3}{2}f_{3}C_{1}$, $C_{19}-\frac{1}{2}f_{1}g_{2}g_{3}C_{1}$,
$C_{20}-\frac{1}{2}f_{1}g_{3}^{2}C_{1}$, $C_{21}-\frac{1}{2}g_{2}C_{1}$,
$C_{22}-\frac{3}{2}f_{2}C_{1}$, $C_{23}-\frac{1}{2}f_{1}g_{2}^{2}C_{1}$,
$C_{24}-\frac{1}{2}f_{1}g_{2}g_{3}C_{1}$. We get the matrix

\begin{center}
\arraycolsep=3pt

{\footnotesize $\left[\begin{array}{cccccccccccc}
0 & 4f_{1}^{3} & 4f_{1}^{2}f_{2} & 4f_{1}^{2}f_{3} & 0 & 4f_{1}f_{2}^{2} & 4f_{2}^{3} & 4f_{2}^{2}f_{3} & 0 & 4f_{1}f_{3}^{2} & 4f_{2}f_{3}^{2} & 4f_{3}^{3}\\
4f_{1} & 0 & 0 & 0 & 0 & 0 & 0 & 0 & 0 & 0 & 0 & 0\\
0 & 0 & 4f_{1}^{2} & 0 & 4f_{2} & 4f_{1}f_{2} & 8f_{2}^{2} & 4f_{2}f_{3} & 0 & 0 & 4f_{3}^{2} & 0\\
0 & 0 & 0 & 4f_{1}^{2} & 0 & 0 & 0 & 4f_{2}^{2} & 4f_{3} & 4f_{1}f_{3} & 4f_{2}f_{3} & 8f_{3}^{2}
\end{array}\right.$}{\footnotesize \par}

{\footnotesize 
\[
\left.\begin{array}{cccccccccccc}
0 & 4f_{1}f_{2}f_{3} & 4f_{2}^{2}f_{3} & 4f_{2}f_{3}^{2} & 0 & 4f_{1}^{2}f_{3} & 4f_{1}f_{2}f_{3} & 4f_{1}f_{3}^{2} & 0 & 4f_{1}^{2}f_{2} & 4f_{1}f_{2}^{2} & 4f_{1}f_{2}f_{3}\\
0 & 0 & 0 & 0 & 0 & 0 & 0 & 0 & 0 & 0 & 0 & 0\\
2f_{3} & 2f_{1}f_{3} & 6f_{2}f_{3} & 2f_{3}^{2} & 0 & 0 & 6f_{1}f_{3} & 0 & 2f_{1} & 2f_{1}^{2} & 6f_{1}f_{2} & 2f_{1}f_{3}\\
2f_{2} & 2f_{1}f_{2} & 2f_{2}^{2} & 6f_{2}f_{3} & 2f_{1} & 2f_{1}^{2} & 2f_{1}f_{2} & 6f_{1}f_{3} & 0 & 0 & 0 & 4f_{1}f_{2}
\end{array}\right].
\]
}{\footnotesize \par}

\end{center}

We further perform the following column operations on the preceding
matrix:

$C_{3}-g_{2}C_{2}$, $C_{4}-g_{3}C_{2}$, $C_{6}-g_{2}^{2}C_{2}$,
$C_{7}-g_{2}^{3}C_{2}$, $C_{8}-g_{2}^{2}g_{3}C_{2}$, $C_{10}-g_{3}^{2}C_{2}$,
$C_{11}-g_{2}g_{3}^{2}C_{2}$, $C_{12}-g_{3}^{3}C_{2}$, $C_{14}-g_{2}g_{3}C_{2}$,
$C_{15}-g_{2}^{2}g_{3}C_{2}$, $C_{16}-g_{2}g_{3}^{2}C_{2}$, $C_{18}-g_{3}C_{2}$,
$C_{19}-g_{2}g_{3}C_{2}$, $C_{20}-g_{3}^{2}C_{2}$, $C_{22}-g_{2}C_{2}$,
$C_{23}-g_{2}^{2}C_{2}$, $C_{24}-g_{2}g_{3}C_{2}$. We get the matrix

\begin{center}
\arraycolsep=3pt

{\footnotesize $\left[\begin{array}{cccccccccccc}
0 & 4f_{1}^{3} & 0 & 0 & 0 & 0 & 0 & 0 & 0 & 0 & 0 & 0\\
4f_{1} & 0 & 0 & 0 & 0 & 0 & 0 & 0 & 0 & 0 & 0 & 0\\
0 & 0 & 4f_{1}^{2} & 0 & 4f_{2} & 4f_{1}f_{2} & 8f_{2}^{2} & 4f_{2}f_{3} & 0 & 0 & 4f_{3}^{2} & 0\\
0 & 0 & 0 & 4f_{1}^{2} & 0 & 0 & 0 & 4f_{2}^{2} & 4f_{3} & 4f_{1}f_{3} & 4f_{2}f_{3} & 8f_{3}^{2}
\end{array}\right.$
\[
\left.\begin{array}{cccccccccccc}
0 & 0 & 0 & 0 & 0 & 0 & 0 & 0 & 0 & 0 & 0 & 0\\
0 & 0 & 0 & 0 & 0 & 0 & 0 & 0 & 0 & 0 & 0 & 0\\
2f_{3} & 2f_{1}f_{3} & 6f_{2}f_{3} & 2f_{3}^{2} & 0 & 0 & 4f_{1}f_{3} & 0 & 2f_{1} & 2f_{1}^{2} & 6f_{1}f_{2} & 2f_{1}f_{3}\\
2f_{2} & 2f_{1}f_{2} & 2f_{2}^{2} & 6f_{2}f_{3} & 2f_{1} & 2f_{1}^{2} & 2f_{1}f_{2} & 6f_{1}f_{3} & 0 & 0 & 0 & 4f_{1}f_{2}
\end{array}\right].
\]
}{\footnotesize \par}

\end{center}

The following column operations are performed on the preceding matrix:

$C_{3}-2f_{1}C_{21}$, $C_{5}-2g_{2}C_{21}$, $C_{6}-2f_{2}C_{21}$,
$C_{7}-4f_{1}g_{2}C_{21}$, $C_{8}-2f_{1}g_{2}g_{3}C_{21}$, $C_{11}-2f_{1}g_{3}^{2}C_{21}$,
$C_{13}-g_{3}C_{21}$, $C_{14}-f_{3}C_{21}$, $C_{15}-3f_{1}g_{2}g_{3}C_{21}$,
$C_{16}-f_{1}g_{3}^{2}C_{21}$, $C_{19}-2f_{3}C_{21}$, $C_{22}-f_{1}C_{21}$,
$C_{23}-3f_{2}C_{21}$, $C_{24}-f_{3}C_{21}$. We get the matrix

\begin{center}
\arraycolsep=3pt

{\footnotesize $\left[\begin{array}{cccccccccccc}
0 & 4f_{1}^{3} & 0 & 0 & 0 & 0 & 0 & 0 & 0 & 0 & 0 & 0\\
4f_{1} & 0 & 0 & 0 & 0 & 0 & 0 & 0 & 0 & 0 & 0 & 0\\
0 & 0 & 0 & 0 & 0 & 0 & 0 & 0 & 0 & 0 & 0 & 0\\
0 & 0 & 0 & 4f_{1}^{2} & 0 & 0 & 0 & 4f_{2}^{2} & 4f_{3} & 4f_{1}f_{3} & 4f_{2}f_{3} & 8f_{3}^{2}
\end{array}\right.$
\[
\left.\begin{array}{cccccccccccc}
0 & 0 & 0 & 0 & 0 & 0 & 0 & 0 & 0 & 0 & 0 & 0\\
0 & 0 & 0 & 0 & 0 & 0 & 0 & 0 & 0 & 0 & 0 & 0\\
0 & 0 & 0 & 0 & 0 & 0 & 0 & 0 & 2f_{1} & 0 & 0 & 0\\
2f_{2} & 2f_{1}f_{2} & 2f_{2}^{2} & 6f_{2}f_{3} & 2f_{1} & 2f_{1}^{2} & 2f_{1}f_{2} & 6f_{1}f_{3} & 0 & 0 & 0 & 4f_{1}f_{2}
\end{array}\right].
\]
}{\footnotesize \par}

\end{center}

Finally, the following column operations are performed on the preceding
matrix: 

$C_{4}-2f_{1}C_{17}$, $C_{8}-2f_{1}g_{2}^{2}C_{17}$, $C_{9}-2g_{3}C_{17}$,
$C_{10}-2f_{3}C_{17}$, $C_{11}-2f_{1}g_{2}g_{3}C_{17}$, $C_{12}-4f_{1}g_{3}^{2}C_{17}$,
$C_{13}-g_{2}C_{17}$, $C_{14}-f_{2}C_{17}$, $C_{15}-f_{1}g_{2}^{2}C_{17}$,
$C_{16}-3f_{1}g_{2}g_{3}C_{17}$, $C_{18}-f_{1}C_{17}$, $C_{19}-f_{2}C_{17}$,
$C_{20}-3f_{3}C_{17}$, $C_{24}-2f_{2}C_{17}$. We get the matrix

\begin{center}
\arraycolsep=3pt

{\footnotesize 
\[
\left[\begin{array}{cccccccccccccccccccccccc}
0 & 4f_{1}^{3} & 0 & 0 & 0 & 0 & 0 & 0 & 0 & 0 & 0 & 0 & 0 & 0 & 0 & 0 & 0 & 0 & 0 & 0 & 0 & 0 & 0 & 0\\
4f_{1} & 0 & 0 & 0 & 0 & 0 & 0 & 0 & 0 & 0 & 0 & 0 & 0 & 0 & 0 & 0 & 0 & 0 & 0 & 0 & 0 & 0 & 0 & 0\\
0 & 0 & 0 & 0 & 0 & 0 & 0 & 0 & 0 & 0 & 0 & 0 & 0 & 0 & 0 & 0 & 0 & 0 & 0 & 0 & 2f_{1} & 0 & 0 & 0\\
0 & 0 & 0 & 0 & 0 & 0 & 0 & 0 & 0 & 0 & 0 & 0 & 0 & 0 & 0 & 0 & 2f_{1} & 0 & 0 & 0 & 0 & 0 & 0 & 0
\end{array}\right].
\]
}{\footnotesize \par}

\end{center}

The columns with all entries zero correspond to the relations among
the columns. A sage script to verfiy these relations is included in
the ancillary files to this paper (verify\_relations.sage). We list
the relations here:
\begin{enumerate}
\item $C_{3}+2f_{2}C_{1}-g_{2}C_{2}-2f_{1}C_{21}=0$.
\item $C_{4}+2f_{3}C_{1}-2f_{1}C_{17}-g_{3}C_{2}=0$.
\item $C_{5}-2g_{2}C_{21}+g_{2}^{2}C_{1}=0$.
\item $C_{6}+2f_{2}g_{2}C_{1}-g_{2}^{2}C_{2}-2f_{2}C_{21}=0$.
\item $C_{7}-g_{2}^{3}C_{2}+4f_{1}g_{2}^{3}C_{1}-4f_{1}g_{2}^{2}C_{21}=0$.
\item $C_{8}-g_{2}^{2}g_{3}C_{2}+4f_{1}g_{2}^{2}g_{3}C_{1}-2f_{1}g_{2}g_{3}C_{21}-2f_{1}g_{2}^{2}C_{17}=0$.
\item $C_{9}-2g_{3}C_{17}+g_{3}^{2}C_{1}=0$.
\item $C_{10}-g_{3}^{2}C_{2}+2f_{1}g_{3}^{2}C_{1}-2f_{3}C_{17}=0$.
\item $C_{11}-g_{2}g_{3}^{2}C_{2}+4f_{1}g_{2}g_{3}^{2}C_{1}-2f_{1}g_{3}^{2}C_{21}-2f_{1}g_{2}g_{3}C_{17}=0$.
\item $C_{12}-g_{3}^{3}C_{2}-4f_{1}g_{3}^{2}C_{17}+4f_{1}g_{3}^{3}C_{1}=0$.
\item $C_{13}-g_{3}C_{21}+g_{2}g_{3}C_{1}-g_{2}C_{17}=0$.
\item $C_{14}-g_{2}g_{3}C_{2}+2f_{1}g_{2}g_{3}C_{1}-f_{3}C_{21}-f_{2}C_{17}=0$.
\item $C_{15}-g_{2}^{2}g_{3}C_{2}-3f_{1}g_{2}g_{3}C_{21}+4f_{1}g_{2}^{2}g_{3}C_{1}-f_{1}g_{2}^{2}C_{17}=0$.
\item $C_{16}-g_{2}g_{3}^{2}C_{2}+4f_{1}g_{2}g_{3}^{2}C_{1}-f_{1}g_{3}^{2}C_{21}-3f_{1}g_{2}g_{3}C_{17}=0$.
\item $C_{18}-g_{3}C_{2}-f_{1}C_{17}+f_{3}C_{1}=0$.
\item $C_{19}+3f_{1}g_{2}g_{3}C_{1}-g_{2}g_{3}C_{2}-2f_{3}C_{21}-f_{2}C_{17}=0$.
\item $C_{20}+3f_{1}g_{3}^{2}C_{1}-g_{3}^{2}C_{2}-3f_{3}C_{17}=0$.
\item $C_{22}-g_{2}C_{2}-f_{1}C_{21}+f_{1}g_{2}C_{1}=0$.
\item $C_{23}+3f_{1}g_{2}^{2}C_{1}-g_{2}^{2}C_{2}-3f_{2}C_{21}=0$.
\item $C_{24}-g_{2}g_{3}C_{2}+3f_{1}g_{2}g_{3}C_{1}-f_{3}C_{21}-2f_{2}C_{17}=0$.
\end{enumerate}
Now consider the first relation $C_{3}+2f_{2}C_{1}-g_{2}C_{2}-2f_{1}C_{21}=0.$
The first column corresponds to coefficients of $4x_{1}f_{1}$, the
second corresponds to coefficients of $(4x_{1}f_{1})x_{1}$, the third
corresponds to coefficients of $(4x_{1}f_{1})x_{2}$, and finally,
the twenty-first column corresponds to coefficients of $2x_{1}f_{2}+2x_{2}f_{1}$.
Now this relation over the columns corresponds to a relation on 
\[
\mathcal{E}=\{4x_{1}f_{1},4x_{2}f_{2},4x_{3}f_{3},2x_{2}f_{3}+2x_{3}f_{2},2x_{1}f_{3}+2x_{3}f_{1},2x_{1}f_{2}+2x_{2}f_{1}\}
\]
 over $R/Q_{1}^{2}$. If $ $$(\alpha_{1},\alpha_{2},\alpha_{3},\alpha_{4},\alpha_{5},\alpha_{6})$
is a relation, then, 
\begin{align*}
\alpha_{i} & =(\text{coefficient of column }(4(i-1)+1))\\
 & +(\text{coefficient of column }(4(i-1)+2))x_{1}\\
 & +(\text{coefficient of column }(4(i-1)+3))x_{2}\\
 & +\text{(coefficient of column }4i)x_{3}
\end{align*}
for $i\in\{1,2,...,6\}$. The relation over $R/Q_{1}^{2}$ corresponding
to the first column relation is then $(2f_{2}-g_{2}x_{1}+x_{2},0,0,0,0,-2f_{1})$.
Proceeding similarly, we get that the relations corresponding to all
column relations. The set of columns of $M$ generate a submodule
of $k[[t]]^{24}$ considered as a $k[[t]]$-module. Since $k[[t]]$
is a principal ideal domain, every submodule is in fact free. Therefore,
the notion of rank is well defined. From the final step after performing
the above column operations on $M$ we can see that $M$ has rank
$4$ and hence the twenty relations above generate the module of relations
on $\{4x_{1}f_{1},4x_{2}f_{2},4x_{3}f_{3},2x_{2}f_{3}+2x_{3}f_{2},2x_{1}f_{3}+2x_{3}f_{1},2x_{1}f_{2}+2x_{2}f_{1}\}$.
We indicate the corresponding relations over $R/Q_{1}^{2}$ in the
following table.

\begin{center}

{\tiny }%
\begin{tabular}{|c|c|}
\hline 
{\tiny Relation over the columns} & {\tiny Relation on $\mathcal{E}$ over $R/Q_{1}^{2}$}\tabularnewline
\hline 
{\tiny $C_{3}+2f_{2}C_{1}-g_{2}C_{2}-2f_{1}C_{21}=0$ } & {\tiny $(2f_{2}-g_{2}x_{1}+x_{2},0,0,0,0,-2f_{1})$}\tabularnewline
\hline 
{\tiny $C_{4}+2f_{3}C_{1}-2f_{1}C_{17}-g_{3}C_{2}=0$} & {\tiny $(2f_{3}-g_{3}x_{1}+x_{3},0,0,0,-2f_{1},0)$}\tabularnewline
\hline 
{\tiny $C_{5}-2g_{2}C_{21}+g_{2}^{2}C_{1}=0$} & {\tiny $(g_{2}^{2},1,0,0,0,-2g_{2})$}\tabularnewline
\hline 
{\tiny $C_{6}+2f_{2}g_{2}C_{1}-g_{2}^{2}C_{2}-2f_{2}C_{21}=0$} & {\tiny $(2f_{2}g_{2}-g_{2}^{2}x_{1},x_{1},0,0,0,-2f_{2})$}\tabularnewline
\hline 
{\tiny $C_{7}-g_{2}^{3}C_{2}+4f_{1}g_{2}^{3}C_{1}-4f_{1}g_{2}^{2}C_{21}=0$} & {\tiny $(4f_{1}g_{2}^{3}-g_{2}^{3}x_{1},x_{2},0,0,0,-4f_{1}g_{2}^{2})$}\tabularnewline
\hline 
{\tiny $C_{8}-g_{2}^{2}g_{3}C_{2}+4f_{1}g_{2}^{2}g_{3}C_{1}-2f_{1}g_{2}g_{3}C_{21}-2f_{1}g_{2}^{2}C_{17}=0$} & {\tiny $(4f_{1}g_{2}^{2}g_{3}-g_{2}^{2}g_{3}x_{1},x_{3},0,0,-2f_{1}g_{2}^{2},-2f_{1}g_{2}g_{3})$}\tabularnewline
\hline 
{\tiny $C_{9}-2g_{3}C_{17}+g_{3}^{2}C_{1}=0$} & {\tiny $(g_{3}^{2},0,1,0,-2g_{3},0)$}\tabularnewline
\hline 
{\tiny $C_{10}-g_{3}^{2}C_{2}+2f_{1}g_{3}^{2}C_{1}-2f_{3}C_{17}=0$} & {\tiny $(2f_{1}g_{3}^{2}-g_{3}^{3}x_{1},0,x_{1},0,-2f_{3},0)$}\tabularnewline
\hline 
{\tiny $C_{11}-g_{2}g_{3}^{2}C_{2}+4f_{1}g_{2}g_{3}^{2}C_{1}-2f_{1}g_{3}^{2}C_{21}-2f_{1}g_{2}g_{3}C_{17}=0$} & {\tiny $(4f_{1}g_{2}g_{3}^{2}-g_{2}g_{3}^{2}x_{1},0,x_{2},0-2f_{1}g_{2}g_{3},-2f_{1}g_{3}^{2})$}\tabularnewline
\hline 
{\tiny $C_{12}-g_{2}^{3}C_{2}-4f_{1}g_{3}^{2}C_{17}+4f_{1}g_{3}^{3}C_{1}=0$} & {\tiny $(4f_{1}g_{3}^{3}-g_{3}^{3}x_{1},0,x_{3},0,-4f_{1}g_{3}^{2},0)$}\tabularnewline
\hline 
{\tiny $C_{13}-g_{3}C_{21}+g_{2}g_{3}C_{1}-g_{2}C_{17}=0$} & {\tiny $(g_{2}g_{3},0,0,1,-g_{2},-g_{3})$}\tabularnewline
\hline 
{\tiny $C_{14}-g_{2}g_{3}C_{2}+2f_{1}g_{2}g_{3}C_{1}-f_{3}C_{21}-f_{2}C_{17}=0$} & {\tiny $(2f_{1}g_{2}g_{3}-g_{2}g_{3}x_{1},0,0,x_{1},-f_{2},-f_{3})$}\tabularnewline
\hline 
{\tiny $C_{15}-g_{2}^{2}g_{3}C_{2}-3f_{1}g_{2}g_{3}C_{21}+4f_{1}g_{2}^{2}g_{3}C_{1}-f_{1}g_{2}^{2}C_{17}=0$} & {\tiny $(4f_{1}g_{2}^{2}g_{3}-g_{2}^{2}g_{3}x_{1},0,0,x_{2},-f_{1}g_{2}^{2},-3f_{1}g_{2}g_{3})$}\tabularnewline
\hline 
{\tiny $C_{16}-g_{2}g_{3}^{2}C_{2}+4f_{1}g_{2}g_{3}^{2}C_{1}-f_{1}g_{3}^{2}C_{21}-3f_{1}g_{2}g_{3}C_{17}=0$} & {\tiny $(4f_{1}g_{2}g_{3}^{2}-g_{2}g_{3}^{2}x_{1},0,0,x_{3},-3f_{1}g_{2}g_{3},-f_{1}g_{3}^{2})$}\tabularnewline
\hline 
{\tiny $C_{18}-g_{3}C_{2}-f_{1}C_{17}+f_{3}C_{1}=0$} & {\tiny $(f_{3}-g_{3}x_{1},0,0,0,-f_{1}+x_{1},0)$}\tabularnewline
\hline 
{\tiny $C_{19}+3f_{1}g_{2}g_{3}C_{1}-g_{2}g_{3}C_{2}-2f_{3}C_{21}-f_{2}C_{17}=0$} & {\tiny $(3f_{1}g_{2}g_{3}-g_{2}g_{3}x_{1},0,0,0,-f_{2}+x_{2},-2f_{3})$}\tabularnewline
\hline 
{\tiny $C_{20}+3f_{1}g_{3}^{2}C_{1}-g_{3}^{2}C_{2}-3f_{3}C_{17}=0$} & {\tiny $(3f_{1}g_{3}^{2}-g_{3}^{2}x_{1},0,0,0,-3f_{3}+x_{3},0)$}\tabularnewline
\hline 
{\tiny $C_{22}-g_{2}C_{2}-f_{1}C_{21}+f_{1}g_{2}C_{1}=0$} & {\tiny $(f_{1}g_{2}-g_{2}x_{1},0,0,0,0,x_{1}-f_{1})$}\tabularnewline
\hline 
{\tiny $C_{23}+3f_{1}g_{2}^{2}C_{1}-g_{2}^{2}C_{2}-3f_{2}C_{21}=0$} & {\tiny $(3f_{1}g_{2}^{2}-g_{2}^{2}x_{1},0,0,0,0,-3f_{2}+x_{2})$}\tabularnewline
\hline 
{\tiny $C_{24}-g_{2}g_{3}C_{2}+3f_{1}g_{2}g_{3}C_{1}-f_{3}C_{21}-2f_{2}C_{17}=0$} & {\tiny $(3f_{1}g_{2}g_{3}-g_{2}g_{3}x_{1},0,0,0,-2f_{2},-f_{3}+x_{3})$}\tabularnewline
\hline 
\end{tabular}{\tiny \par}

\end{center}

\bigskip{}

Now the relation $(\alpha_{1},\alpha_{2},\alpha_{3},\alpha_{4},\alpha_{5},\alpha_{6})$
corresponds to the generator of $Q_{1}^{2}\cap Q_{2}^{2}$ given by
$\alpha_{1}(x_{1}+f_{1})^{2}+\alpha_{2}(x_{2}+f_{2})^{2}+\alpha_{3}(x_{3}+f_{3})^{2}+\alpha_{4}(x_{2}+f_{2})(x_{3}+f_{3})+\alpha_{5}(x_{1}+f_{1})(x_{3}+f_{3})+\alpha_{6}(x_{1}+f_{1})(x_{2}+f_{2})$.
Corresponding to the above relations we obtain generators of $Q_{1}^{2}\cap Q_{2}^{2}$
as follows: 
\begin{enumerate}
\item $\gamma_{1}=(2f_{2}-g_{2}x_{1}+x_{2})(x_{1}+f_{1})^{2}-2f_{1}$$(x_{1}+f_{1})(x_{2}+f_{2})$.
\item $\gamma_{2}=(2f_{3}-g_{3}x_{1}+x_{3})(x_{1}+f_{1})^{2}-2f_{1}(x_{1}+f_{1})(x_{3}+f_{3})$.
\item $\gamma_{3}=g_{2}^{2}(x_{1}+f_{1})^{2}+(x_{2}+f_{2})^{2}-2g_{2}(x_{1}+f_{1})(x_{2}+f_{2})$.
\item $\gamma_{4}=(2f_{2}g_{2}-g_{2}^{2}x_{1})(x_{1}+f_{1})^{2}+x_{1}(x_{2}+f_{2})^{2}-2f_{2}$$(x_{1}+f_{1})(x_{2}+f_{2})$.
\item $\gamma_{5}=(4f_{1}g_{2}^{3}-g_{2}^{3}x_{1})(x_{1}+f_{1})^{2}+x_{2}(x_{2}+f_{2})^{2}-4f_{1}g_{2}^{2}$$(x_{1}+f_{1})(x_{2}+f_{2})$.
\item $\gamma_{6}=(4f_{1}g_{2}^{2}g_{3}-g_{2}^{2}g_{3}x_{1})(x_{1}+f_{1})^{2}+x_{3}(x_{2}+f_{2})^{2}-2f_{1}g_{2}^{2}(x_{1}+f_{1})(x_{3}+f_{3})-2f_{1}g_{2}g_{3}(x_{1}+f_{1})(x_{2}+f_{2})$.
\item $\gamma_{7}=g_{3}^{2}(x_{1}+f_{1})^{2}+(x_{3}+f_{3})^{2}-2g_{3}(x_{1}+f_{1})(x_{3}+f_{3})$.
\item $\gamma_{8}=(2f_{1}g_{3}^{2}-g_{3}^{2}x_{1})(x_{1}+f_{1})^{2}+x_{1}(x_{3}+f_{3})^{2}-2f_{3}(x_{1}+f_{1})(x_{3}+f_{3})$.
\item $\gamma_{9}=(4f_{1}g_{2}g_{3}^{2}-g_{2}g_{3}^{2}x_{1})(x_{1}+f_{1})^{2}+x_{2}(x_{3}+f_{3})^{2}-2f_{1}g_{2}g_{3}(x_{1}+f_{1})(x_{3}+f_{3})-2f_{1}g_{3}^{2}(x_{1}+f_{1})(x_{2}+f_{2})$.
\item $\gamma_{10}=(4f_{1}g_{3}^{3}-g_{3}^{3}x_{1})(x_{1}+f_{1})^{2}+x_{3}(x_{3}+f_{3})^{2}-4f_{1}g_{3}^{2}((x_{1}+f_{1})(x_{3}+f_{3})$.
\item $\gamma_{11}=g_{2}g_{3}(x_{1}+f_{1})^{2}+(x_{2}+f_{2})(x_{3}+f_{3})-g_{2}(x_{1}+f_{1})(x_{3}+f_{3})-g_{3}(x_{1}+f_{1})(x_{2}+f_{2})$.
\item $\gamma_{12}=(2f_{1}g_{2}g_{3}-g_{2}g_{3}x_{1})(x_{1}+f_{1})^{2}+x_{1}(x_{2}+f_{2})(x_{3}+f_{3})-f_{2}(x_{1}+f_{1})(x_{3}+f_{3})-f_{3}(x_{1}+f_{1})(x_{2}+f_{2})$.
\item $\gamma_{13}=(4f_{1}g_{2}^{2}g_{3}-g_{2}^{2}g_{3}x_{1})(x_{1}+f_{1})^{2}+x_{2}(x_{2}+f_{2})(x_{3}+f_{3})-f_{1}g_{2}^{2}(x_{1}+f_{1})(x_{3}+f_{3})-3f_{1}g_{2}g_{3}(x_{1}+f_{1})(x_{2}+f_{2})$.
\item $\gamma_{14}=(4f_{1}g_{2}g_{3}^{2}-g_{2}g_{3}^{2}x_{1})(x_{1}+f_{1})^{2}+x_{3}(x_{2}+f_{2})(x_{3}+f_{3})-3f_{1}g_{2}g_{3}(x_{1}+f_{1})(x_{3}+f_{3})-f_{1}g_{3}^{2}((x_{1}+f_{1})(x_{2}+f_{2})$.
\item $\gamma_{15}=(f_{3}-g_{3}x_{1})(x_{1}+f_{1})^{2}+(x_{1}-f_{1})(x_{1}+f_{1})(x_{3}+f_{3})$.
\item $\gamma_{16}=(3f_{1}g_{2}g_{3}-g_{2}g_{3}x_{1})(x_{1}+f_{1})^{2}+(x_{2}-f_{2})(x_{1}+f_{1})(x_{3}+f_{3})-2f_{3}(x_{1}+f_{1})(x_{2}+f_{2})$.
\item $\gamma_{17}=(3f_{1}g_{3}^{2}-g_{3}^{2}x_{1})(x_{1}+f_{1})^{2}+(x_{3}-3f_{3})(x_{1}+f_{1})(x_{3}+f_{3})$.
\item $\gamma_{18}=(f_{1}g_{2}-g_{2}x_{1})(x_{1}+f_{1})^{2}+(x_{1}-f_{1})(x_{1}+f_{1})(x_{2}+f_{2})$.
\item $\gamma_{19}=(3f_{1}g_{2}^{2}-g_{2}^{2}x_{1})(x_{1}+f_{1})^{2}+(x_{2}-3f_{2})(x_{1}+f_{1})(x_{2}+f_{2})$.
\item $\gamma_{20}=(3f_{1}g_{2}g_{3}-g_{2}g_{3}x_{1})(x_{1}+f_{1})^{2}-2f_{2}((x_{1}+f_{1})(x_{3}+f_{3})+(x_{3}-f_{3})(x_{1}+f_{1})(x_{2}+f_{2})$.
\end{enumerate}
Since these generators correspond to relations on $\mathcal{E=}\{4x_{1}f_{1},4x_{2}f_{2},4x_{3}f_{3},2x_{2}f_{3}+2x_{3}f_{2},2x_{1}f_{3}+2x_{3}f_{1},2x_{1}f_{2}+2x_{2}f_{1}\}$
over $R/Q_{1}^{2}$, we need to include the set of generators of $Q_{1}^{2}Q_{2}^{2}$
to get a complete set of generators for $Q_{1}^{2}\cap Q_{2}^{2}$.
Thus, $Q_{1}^{2}\cap Q_{2}^{2}=(\{\gamma_{i}:i=1,...,20\})R+Q_{1}^{2}Q_{2}^{2}$.

\subsection{$(Q_{1}^{2}\cap Q_{2}^{2})\cap S\subseteq\mathfrak{m}P$ (special
case)\label{sub:-(special-case)}}

We obtain the contraction of $Q_{1}^{2}\cap Q_{2}^{2}$ to $S$ by
applying the trace map to each of the generators of $Q_{1}^{2}\cap Q_{2}^{2}$.
Showing that each of these generators of $(Q_{1}^{2}\cap Q_{2}^{2})\cap S$
lies in $\mathfrak{m}P$ proves the Eisenbud-Mazur conjecture in this
case. Note that we showed earlier $P=(\{(-x_{i}+x_{1}g_{i}(t)):i=2,...,m\})S+(\{(x_{i}x_{j}-f_{i}(t)f_{j}(t)):i,j\in\{1,...,m\}\})S$.
In the case of three generators, we have that $P=(x_{1}g_{2}-x_{2},x_{1}g_{3}-x_{3},x_{1}x_{2}-f_{1}f_{2},x_{1}x_{3}-f_{1}f_{3},x_{2}x_{3}-f_{2}f_{3},x_{1}^{2}-f_{1}^{2},x_{2}^{2}-f_{2}^{2},x_{3}^{2}-f_{3}^{2})S$.
We can express each generator of $ $$(Q_{1}^{2}\cap Q_{2}^{2})\cap S$
as a linear combination of these generators of $P$ with coefficients
in $\mathfrak{m}$. These computations are a special case of the computations
for the general case illustrated in Section \ref{sub: general case computations}.

\subsection{Computing generators of $Q_{1}^{2}\cap Q_{2}^{2}$ (general case)}

We extend the methods of section \ref{special case q12 cap q22} to
compute the generators of $Q_{1}^{2}\cap Q_{2}^{2}$ for any number
$m$ of generators of $Q_{1}$. In this case, we seek relations on
\begin{align*}
\mathcal{E} & =(4x_{1}f_{1},...,4x_{m}f_{m},2x_{1}f_{2}+2x_{2}f_{1},...,2x_{1}f_{m}+2x_{m}f_{1},2x_{2}f_{3}+2x_{3}f_{2},...,\\
 & 2x_{2}f_{m}+2x_{m}f_{2},...,2x_{m-1}f_{m}+2x_{m}f_{m-1})
\end{align*}
 over $R/Q_{1}^{2}$. Denote a typical relation by the vector 
\[
(\alpha_{1},...,\alpha_{m},\alpha_{12},...,\alpha_{1m},\alpha_{23},...,\alpha_{2m},...,\alpha_{m-1,m}).
\]
 As before, we can express every element of $R/Q_{1}^{2}$ as a $k[[t]]\text{-}$linear
combination of $x_{1},...,x_{m}$. We associate the matrix $M$ of
coefficients in $k[[t]]$ with the set $\mathcal{E}$ as we did in
the case when $m=3$. Corresponding to each element in $\mathcal{E}$,
we will have $m+1$ columns of coefficients. Now $|\mathfrak{\mathcal{E}|}=m+\frac{m(m-1)}{2}=\frac{m(m+1)}{2}$.
So we have that $\frac{m(m+1)(m+1)}{2}$ columns in $M$. Performing
column operations as before, we can show that $M$ has rank equal
to the number of rows, which is $m+1$, giving rise to $\frac{m(m+1)^{2}}{2}-(m+1)=\frac{(m-1)(m+1)(m+2)}{2}$
relations. We indicate these relations below. 
\begin{enumerate}
\item The $m-1$ relations with $\alpha_{1}=g_{i}^{2},\alpha_{i}=1,\alpha_{1i}=-2g_{i}$,
$1<i\leq m$.
\item The $m-1$ relations with $\alpha_{1}=g_{i}(f_{1}-x_{1}),\alpha_{1i}=x_{1}-f_{1}$,
$1<i\leq m$.
\item The $(m-1)(m-2)$ relations with $\alpha_{1}=g_{i}g_{j}(3f_{1}-x_{1}),\alpha_{1i}=x_{j}-f_{j},\alpha_{1j}=-2f_{i}$,
$1<i\neq j<m$.
\item The $(m-1)(m-2)/2$ relations with $\alpha_{1}=g_{i}g_{j}(2f_{1}-x_{1}),\alpha_{1i}=-f_{j},\alpha_{1j}=-f_{i},\alpha_{ij}=x_{1}$,
$1<i\neq j<m$.
\item The $(m-1)(m-2)/2$ relations with $\alpha_{1}=g_{i}g_{j},\alpha_{1i}=-g_{j},\alpha_{1j}=-g_{i},\alpha_{ij}=1$,
$1<i\neq j<m$.
\item The $(m-1)(m-2)$ relations with $\alpha_{1}=g_{i}^{2}g_{j}(4f_{1}-x_{1}),\alpha_{1i}=-3f_{1}g_{i}g_{j},\alpha_{1j}=-f_{1}g_{i}^{2},\alpha_{ij}=x_{i}$,
$1<i\neq j<m$.
\item The $m-1$ relations with $\alpha_{1}=g_{i}^{2}(3f_{1}-x_{1}),\alpha_{1i}=x_{i}-3f_{i}$,
$1<i\leq m$.
\item The $m-1$ relations with $\alpha_{1}=g_{i}(2f_{1}-x_{1})+x_{i},\alpha_{1i}=-2f_{1}$,
$1<i\leq m$.
\item The $m-1$ relations with $\alpha_{1}=g_{i}^{3}(4f_{1}-x_{1}),\alpha_{i}=x_{i},\alpha_{1i}=-4f_{1}g_{i}^{2}$,
$1<i\leq m$.
\item The $m-1$ relations with $\alpha_{1}=g_{i}^{2}(2f_{1}-x_{1}),\alpha_{i}=x_{1},\alpha_{1i}=-2f_{i}$,
$1<i\leq m$.
\item The $(m-1)(m-2)$ relations with $\alpha_{1}=g_{i}^{2}g_{j}(4f_{1}-x_{1}),\alpha_{i}=x_{j},\alpha_{1i}=-2f_{1}g_{i}g_{j},\alpha_{1j}=-2f_{1}g_{i}^{2}$,
$1<i\neq j<m$.
\item The $(m-1)(m-2)(m-3)/2$ relations with
\end{enumerate}
$\begin{cases}
\alpha_{1}=g_{i}g_{j}g_{k}(2f_{1}-x_{1}),\alpha_{ij}=x_{k},\alpha_{ki}=-f_{1}g_{j},\alpha_{kj}=-f_{1}g_{i}, & 1<k<i<j\leq n\\
\alpha_{1}=g_{i}g_{j}g_{k}(2f_{1}-x_{1}),\alpha_{ij}=x_{k},\alpha_{ik}=-f_{1}g_{j},\alpha_{kj}=-f_{1}g_{i}, & 1<i<k<j\leq n\\
\alpha_{1}=g_{i}g_{j}g_{k}(2f_{1}-x_{1}),\alpha_{ij}=x_{k},\alpha_{ik}=-f_{1}g_{j},\alpha_{jk}=-f_{1}g_{i}, & 1<i<j<k\leq n
\end{cases}$.

So we have accounted for all the $6(m-1)+4(m-1)(m-2)+\frac{(m-1)(m-2)(m-3)}{2}=\frac{(m-1)(m+1)(m+2)}{2}$
relations. 

The corresponding generators of $Q_{1}^{2}\cap Q_{2}^{2}$ are as
follows.
\begin{enumerate}
\item The $m-1$ generators $g_{i}^{2}(x_{1}+f_{1})^{2}+(x_{i}+f_{i})^{2}-2g_{i}(x_{1}+f_{1})(x_{i}+f_{i})$,
$1<i\leq m$.
\item The $m-1$ generators $g_{i}(f_{1}-x_{1})(x_{1}+f_{1})^{2}+(x_{1}-f_{1})(x_{1}+f_{1})(x_{i}+f_{i})$,
$1<i\leq m$.
\item The $(m-1)(m-2)$ generators $g_{i}g_{j}(3f_{1}-x_{1})(x_{1}+f_{1})^{2}+(x_{j}-f_{j})(x_{1}+f_{1})(x_{i}+f_{i})+-2f_{i}(x_{1}+f_{1})(x_{j}+f_{j})$,
$1<i\neq j<m$.
\item The $(m-1)(m-2)/2$ generators $g_{i}g_{j}(2f_{1}-x_{1})(x_{1}+f_{1})^{2}-f_{j}(x_{1}+f_{1})(x_{i}+f_{i})-f_{i}(x_{1}+f_{1})(x_{j}+f_{j})+x_{1}(x_{i}+f_{i})(x_{j}+f_{j})$,
$1<i\neq j<m$.
\item The $(m-1)(m-2)/2$ generators $g_{i}g_{j}(x_{1}+f_{1})^{2}-g_{j}(x_{1}+f_{1})(x_{i}+f_{i})-g_{i}(x_{1}+f_{1})(x_{j}+f_{j})+(x_{i}+f_{i})(x_{j}+f_{j})$,
$1<i\neq j<m$.
\item The $(m-1)(m-2)$ generators $g_{i}^{2}g_{j}(4f_{1}-x_{1})(x_{1}+f_{1})^{2}-3f_{1}g_{i}g_{j}(x_{1}+f_{1})(x_{i}+f_{i})-f_{1}g_{i}^{2}(x_{1}+f_{1})(x_{j}+f_{j})+x_{i}(x_{i}+f_{i})(x_{j}+f_{j})$,
$1<i\neq j<m$.
\item The $m-1$ generators $g_{i}^{2}(3f_{1}-x_{1})(x_{1}+f_{1})^{2}+(x_{i}-3f_{i})(x_{1}+f_{1})(x_{i}+f_{i})$,
$1<i\leq m$.
\item The $m-1$ generators $(g_{i}(2f_{1}-x_{1})+x_{i})(x_{1}+f_{1})^{2}-2f_{1}(x_{1}+f_{1})(x_{i}+f_{i})$,
$1<i\leq m$.
\item The $m-1$ generators $g_{i}^{3}(4f_{1}-x_{1})(x_{1}+f_{1})^{2}+x_{i}(x_{i}+f_{i})^{2}-4f_{1}g_{i}^{2}(x_{1}+f_{1})(x_{i}+f_{i})$,
$1<i\leq m$.
\item The $m-1$ generators $g_{i}^{2}(2f-x_{1})(x_{1}+f_{1})^{2}+x_{1}(x_{i}+f_{i})^{2}-2f_{i}(x_{1}+f_{1})(x_{i}+f_{i})$,
$1<i\leq m$.
\item The $(m-1)(m-2)$ generators $g_{i}^{2}g_{j}(4f_{1}-x_{1})(x_{1}+f_{1})^{2}+x_{j}(x_{i}+f_{i})^{2}-2f_{1}g_{i}g_{j}(x_{1}+f_{1})(x_{i}+f_{i})-2f_{1}g_{i}^{2}(x_{1}+f_{1})(x_{j}+f_{j})$,
$1<i\neq j<m$.
\item The $(m-1)(m-2)(m-3)/2$ generators $g_{i}g_{j}g_{k}(2f_{1}-x_{1})(x_{1}+f_{1})^{2}+x_{k}(x_{i}+f_{i})(x_{j}+f_{j})-f_{1}g_{j}(x_{k}+f_{k})(x_{i}+f_{i})-f_{1}g_{j}(x_{k}+f_{k})(x_{j}+f_{j})$.
\end{enumerate}

\subsection{$(Q_{1}^{2}\cap Q_{2}^{2})\cap S\subseteq\mathfrak{m}P$ (general
case)\label{sub: general case computations}}

As discussed in Section \ref{sub:-(special-case)}, in the special
case of three generators, we need to apply the trace map to each of
the generators of $Q_{1}^{2}\cap Q_{2}^{2}$ above along with the
generators of $Q_{1}^{2}Q_{2}^{2}$ to get the generators for $(Q_{1}^{2}\cap Q_{2}^{2})\cap S$
and show that each of these generators lies in $\mathfrak{m}P$. This
will complete the proof that $P^{(2)}\subseteq\mathfrak{m}P$ in the
general case. We illustrate these computations below. A step by step
verification of these computations is available as an ancillary file
to this paper (appendix\_of\_computations.pdf).
\begin{enumerate}
\item $\text{tr}(g_{i}^{2}(x_{1}+f_{1})^{2}+(x_{i}+f_{i})^{2}-2g_{i}((x_{1}+f_{1})(x_{i}+f_{i})))=x_{1}^{2}g_{i}^{2}+f_{1}^{2}g_{i}^{2}-2x_{1}x_{i}g_{i}-2f_{1}f_{2}g_{i}+x_{i}^{2}+f_{i}^{2}=x_{1}g_{i}(x_{1}g_{i}-x_{i})-x_{i}(x_{1}g_{i}-x_{i})\in\mathfrak{m}P$.
\item $\text{tr}(g_{i}(f_{1}-x_{1})(x_{1}+f_{1})^{2}+(x_{1}-f_{1})(x_{1}+f_{1})(x_{i}+f_{i}))=-x_{1}^{3}g_{i}+x_{1}^{2}f_{1}g_{i}-x_{1}f_{1}^{2}g_{i}+f_{1}^{3}g_{i}+x_{1}^{2}x_{i}-x_{1}x_{i}f_{1}+x_{1}f_{1}f_{i}-f_{1}^{2}f_{i}=-x_{1}^{2}(x_{1}g_{i}-x_{i})+x_{1}f_{1}(x_{1}g_{i}-x_{i})\in\mathfrak{m}P$.
\item $\text{tr}(g_{i}g_{j}(3f_{1}-x_{1})(x_{1}+f_{1})^{2}-2f_{i}((x_{1}+f_{1})(x_{j}+f_{j})+(x_{j}-f_{j})((x_{1}+f_{1})(x_{i}+f_{i}))=-x_{1}^{3}g_{i}g_{j}+3x_{1}^{2}f_{1}g_{i}g_{j}-x_{1}f_{1}^{2}g_{i}g_{j}+3f_{1}^{3}g_{i}g_{j}-2x_{1}x_{j}f_{1}g_{i}-2f_{1}^{2}f_{j}g_{i}-x_{1}x_{i}f_{1}g_{j}-f_{1}^{2}f_{i}g_{j}+x_{1}x_{i}x_{j}+x_{j}f_{1}f_{i}=-x_{1}^{3}g_{j}(x_{1}g_{i}-x_{i})-x_{1}x_{i}(x_{1}g_{j}-x_{j})+2x_{1}f_{1}g_{i}(x_{1}g_{j}-x_{j})+x_{1}f_{1}g_{j}(x_{1}g_{i}-x_{i})+f_{1}^{2}g_{i}(x_{j}-x_{1}g_{j})\in\mathfrak{m}P$.
\item $\text{tr}(g_{i}g_{j}(2f_{1}-x_{1})(x_{1}+f_{1})^{2}+x_{1}(x_{i}+f_{i})(x_{j}+f_{j})-f_{i}(x_{1}+f_{1})(x_{j}+f_{j})-f_{j}(x_{1}+f_{1})(x_{i}+f_{i}))=-x_{1}^{3}g_{i}g_{j}+2x_{1}^{2}f_{1}g_{i}g_{j}-x_{1}f_{1}^{2}g_{i}g_{j}+2f_{1}^{3}g_{i}g_{j}+x_{1}x_{i}x_{j}-x_{1}x_{j}f_{i}-x_{1}x_{i}f_{j}+x_{1}f_{i}f_{j}-2f_{1}f_{i}f_{j}=-x_{1}^{3}g_{j}(x_{1}g_{i}-x_{i})-x_{1}x_{i}(x_{1}g_{j}-x_{j})+x_{1}f_{1}g_{j}(x_{1}g_{i}-x_{i})+x_{1}f_{1}g_{i}(x_{1}g_{j}-x_{j})\in\mathfrak{m}P$.
\item $\text{tr}(g_{i}g_{j}(x_{1}+f_{1})^{2}+(x_{i}+f_{i})(x_{j}+f_{j})-g_{i}(x_{1}+f_{1})(x_{j}+f_{j})-g_{j}(x_{1}+f_{1})(x_{i}+f_{i}))=x_{1}^{2}g_{i}g_{j}+f_{1}^{2}g_{i}g_{j}-x_{1}x_{j}g_{i}-f_{1}f_{j}g_{i}-x_{1}x_{i}g_{j}-f_{1}f_{i}g_{j}+x_{i}x_{j}+f_{i}f_{j}=x_{1}g_{i}(x_{1}g_{j}-x_{j})-x_{i}(x_{1}g_{j}-x_{j})\in\mathfrak{m}P$.
\item $\text{tr}(g_{i}^{2}g_{j}(4f_{1}-x_{1})(x_{1}+f_{1})^{2}+x_{i}(x_{i}+f_{i})(x_{j}+f_{j})-f_{1}g_{i}^{2}(x_{1}+f_{1})(x_{j}+f_{j})-3f_{1}g_{i}g_{j}(x_{1}+f_{1})(x_{j}+f_{j}))=-x_{1}^{3}g_{i}^{2}g_{j}+4x_{1}^{2}f_{1}g_{i}^{2}g_{j}-x_{1}f_{1}^{2}g_{i}^{2}g_{j}+4f_{1}^{3}g_{i}^{2}g_{j}-x_{1}x_{j}f_{1}g_{i}^{2}-f_{1}^{2}f_{j}g_{i}^{2}-3x_{1}x_{i}f_{1}g_{i}g_{j}-3f_{1}^{2}f_{i}g_{i}g_{j}+x_{i}^{2}x_{j}+x_{i}f_{i}f_{j}=(-x_{1}^{2}g_{j}g_{i}-x_{1}x_{i}g_{j})(x_{1}g_{i}-x_{i})-x_{i}^{2}(x_{1}g_{j}-x_{j})+x_{1}f_{1}g_{i}^{2}(x_{1}g_{j}-x_{j})+3x_{1}f_{1}g_{i}g_{j}(x_{1}g_{i}-x_{i})+f_{1}^{2}g_{i}g_{j}(x_{i}-x_{1}g_{i})\in\mathfrak{m}P$.
\item $\text{tr}(g_{i}^{2}(3f_{1}-x_{1})(x_{1}+f_{1})^{2}+(x_{i}-3f_{i})(x_{1}+f_{1})(x_{i}+f_{i}))=-x_{1}^{3}g_{i}^{2}+3x_{1}^{2}f_{1}g_{i}^{2}-x_{1}f_{1}^{2}g_{i}^{2}+3f_{1}^{3}g_{i}^{2}-3x_{1}x_{i}f_{1}g_{i}-3f_{1}^{2}f_{i}g_{i}+x_{1}x_{i}^{2}+x_{i}f_{1}f_{i}=-x_{1}(x_{1}g_{i}+x_{i})(x_{1}g_{i}-x_{i})+3x_{1}f_{1}g_{i}(x_{1}g_{i}-x_{i})+f_{1}^{2}g_{i}(x_{i}-x_{1}g_{i})\in\mathfrak{m}P$.
\item $\text{tr}((g_{i}(2f_{1}-x_{1})+x_{i})(x_{1}+f_{1})^{2}-2f_{1}((x_{1}+f_{1})(x_{i}+f_{i})))=-x_{1}^{3}g_{i}+x_{1}^{2}x_{i}-2x_{1}^{2}f_{1}g_{i}+2x_{1}^{2}f_{i}-x_{1}f_{1}^{2}g_{i}-x_{i}f_{1}^{2}+2x_{1}f_{1}f_{i}=-x_{1}^{2}(g_{i}x_{1}-x_{i})+f_{1}^{2}(g_{i}x_{1}-x_{i})\in\mathfrak{m}P$.
\item $\text{tr}(g_{i}^{3}(4f_{1}-x_{1})(x_{1}+f_{1})^{2}+x_{i}(x_{i}+f_{i})^{2}-4f_{1}g_{i}^{2}(x_{1}+f_{1})(x_{i}+f_{i}))=-x_{1}^{3}g_{i}^{3}+4x_{1}^{2}f_{1}g_{i}^{3}-x_{1}f_{1}^{2}g_{i}^{3}+4f_{1}^{3}g_{i}^{3}-4x_{1}x_{i}f_{1}g_{i}^{2}-4f_{1}^{2}f_{i}g_{i}^{2}+x_{i}^{3}+x_{i}f_{i}^{2}=(x_{i}-x_{1}g_{i})(x_{i}^{2}+x_{1}x_{i}g_{i}+x_{1}^{2}g_{i}^{2})-f_{1}^{2}g_{i}^{2}(x_{1}g_{i}-x_{i})+4x_{1}f_{1}g_{i}^{2}(x_{1}g_{i}-x_{i})\in\mathfrak{m}P$.
\item $\text{tr}(g_{i}^{2}(2f-x_{1})(x_{1}+f_{1})^{2}+x_{1}(x_{i}+f_{i})^{2}-2f_{i}((x_{1}+f_{1})(x_{i}+f_{i})))=-x_{1}^{3}g_{i}^{2}-x_{1}f_{1}^{2}g_{i}^{2}+2x_{1}^{2}f_{i}g_{i}+2f_{1}^{2}f_{i}g_{i}+x_{1}x_{i}^{2}-2x_{1}x_{i}f_{i}+x_{1}f_{i}^{2}-2f_{1}f_{i}^{2}=-x_{1}(x_{1}g_{i}+x_{i})(x_{1}g_{i}-x_{i})-2x_{1}f_{1}g_{i}(x_{i}-x_{1}g_{i})\in\mathfrak{m}P$.
\item $\text{tr}(g_{i}^{2}g_{j}(4f_{1}-x_{1})(x_{1}+f_{1})^{2}+x_{j}(x_{i}+f_{i})^{2}-2f_{1}g_{i}^{2}(x_{1}+f_{1})(x_{j}+f_{j})-2f_{1}g_{i}g_{j}(x_{1}+f_{1})(x_{i}+f_{i}))=-x_{1}^{3}g_{i}^{2}g_{j}+4x_{1}^{2}f_{1}g_{i}^{2}g_{j}-x_{1}f_{1}^{2}g_{i}^{2}g_{j}+4f_{1}^{3}g_{i}^{2}g_{j}-2x_{1}x_{j}f_{1}g_{i}^{2}-2f_{1}^{2}f_{j}g_{i}^{2}-2x_{1}x_{i}f_{1}g_{i}g_{j}-2f_{1}^{2}f_{i}g_{i}g_{j}+x_{i}^{2}x_{j}+x_{j}f_{i}^{2}=-x_{1}^{2}g_{i}g_{j}(x_{1}g_{i}-x_{i})-x_{1}x_{i}g_{j}(x_{1}g_{i}-x_{i})-x_{i}^{2}(x_{1}g_{j}-x_{j})-f_{1}^{2}g_{i}^{2}(x_{1}g_{j}-x_{j})+2x_{1}f_{1}g_{i}^{2}(x_{1}g_{j}-x_{j})+2x_{1}f_{1}g_{i}g_{j}(x_{1}g_{i}-x_{i})\in\mathfrak{m}P$.
\item $\text{tr}(g_{i}g_{j}g_{k}(2f_{1}-x_{1})(x_{1}+f_{1})^{2}+x_{k}(x_{i}+f_{i})(x_{j}+f_{j})-f_{1}g_{j}(x_{k}+f_{k})(x_{i}+f_{i})-f_{1}g_{j}(x_{k}+f_{k})(x_{j}+f_{j}))=-x_{1}^{3}g_{i}g_{j}g_{k}+2x_{1}^{2}f_{1}g_{1}g_{j}g_{k}-x_{1}f_{1}^{2}g_{i}g_{j}g_{k}+2f_{1}^{3}g_{i}g_{j}g_{k}+x_{i}x_{j}x_{k}-f_{i}x_{j}x_{k}-f_{j}x_{i}x_{k}+f_{i}f_{j}x_{k}-2f_{i}f_{j}f_{k}=-x_{1}^{2}g_{i}g_{j}(x_{1}g_{k}-x_{k})-x_{1}x_{k}g_{i}(x_{1}g_{j}-x_{j})-x_{j}x_{k}(x_{1}g_{i}-x_{i})-f_{1}^{2}g_{i}g_{j}(x_{1}g_{k}-x_{k})+f_{1}g_{i}g_{j}x_{1}(x_{1}g_{k}-x_{k})+f_{1}g_{i}x_{k}(x_{1}g_{j}-x_{j})+f_{1}g_{i}g_{j}x_{1}(x_{1}g_{k}-x_{k})+f_{1}g_{j}x_{k}(x_{1}g_{i}-x_{i})\in\mathfrak{m}P$.
\end{enumerate}
Lastly, we need to show the generators of $(Q_{1}^{2}\cap Q_{2}^{2})\cap S$
arising by applying the trace map to $Q_{1}^{2}Q_{2}^{2}$ also lie
in $\mathfrak{m}P$. We can express, 
\[
Q_{1}^{2}Q_{2}^{2}=\{(x_{i}-f_{i})(x_{j}-f_{j})(x_{i^{'}}+f_{i^{'}})(x_{j^{'}}+f_{j^{'}}):i,i^{'},j,j^{'}\in\{1,...,m\}\}.
\]

Then we have that $\text{tr}((x_{i}-f_{i})(x_{j}-f_{j})(x_{i^{'}}+f_{i^{'}})(x_{j^{'}}+f_{j^{'}}))=x_{i}x_{i^{'}}x_{j}x_{j^{'}}+f_{i}f_{i^{'}}x_{j}x_{j^{'}}+x_{i}x_{i^{'}}f_{j}f_{j^{'}}-f_{i}x_{i^{'}}f_{j}x_{j^{'}}-f_{i}x_{i^{'}}x_{j}f_{j^{'}}-x_{i}f_{i^{'}}f_{j}x_{j^{'}}-x_{i}f_{i^{'}}x_{j}f_{j^{'}}+f_{i}f_{i^{'}}f_{j}f_{j^{'}}=x_{i^{'}}x_{j^{'}}(x_{i}x_{j}-f_{i}f_{j})+x_{j}x_{j^{'}}(f_{i}f_{i^{'}}-x_{i}x_{i^{'}})+x_{i}x_{j^{'}}(x_{i^{'}}x_{j}-f_{i^{'}}f_{j})+x_{i}x_{i^{'}}(f_{j}f_{j^{'}}-x_{j}x_{j^{'}})+x_{i^{'}}x_{j}(x_{i}x_{j^{'}}-f_{i}f_{j^{'}})-f_{i^{'}}f_{j^{'}}(x_{i}x_{j}-f_{i}f_{j})\in\mathfrak{m}P$.

Thus, every generator of $(Q_{1}^{2}\cap Q_{2}^{2})\cap S$ is in
$\mathfrak{m}P$, showing $(Q_{1}^{2}\cap Q_{2}^{2})\cap S\subseteq\mathfrak{m}P$
as promised. Thus, we have shown that $P^{(2)}\subseteq\mathfrak{m}P$
(using corollary \ref{cor:symbolic power containment}), completing
the proof of the main theorem.
\begin{acknowledgement*}
I would like to thank Prof. Melvin Hochster for his ideas and advice
during the preparation of this manuscript, which is part of my doctoral
thesis submitted to the University of Michigan, Ann Arbor.
\end{acknowledgement*}
\bibliographystyle{elsarticle-harv}
\bibliography{thesisref}

\end{document}